\documentclass[a4paper,leqno,12pt]{amsart}

\usepackage{ulem}
\usepackage{enumerate}
\usepackage{amssymb,amsthm,amsmath}
\usepackage{mathrsfs}
\usepackage[colorlinks,breaklinks]{hyperref}
\usepackage[margin=1.4in]{geometry}
\usepackage{braket}
\usepackage{comment}
\usepackage{tikz}
\usepackage{pgfplots}
\usepackage{graphicx}
\usepackage{comment}
\usepackage{caption}
\pgfplotsset{compat=1.11}
\usepgfplotslibrary{fillbetween}
\usetikzlibrary{intersections}

\usetikzlibrary{positioning,arrows}
\usepackage{capt-of}
\usepackage{caption}
\usepackage{float}
\usetikzlibrary{shapes}
\usetikzlibrary{plotmarks}
\tikzset{
  state/.style={circle,draw,minimum size=6ex},
  arrow/.style={-latex, shorten >=1ex, shorten <=1ex}}

\theoremstyle{plain}
\newtheorem{prop}{Proposition}[section]
\newtheorem{lemma}[prop]{Lemma}
\newtheorem{theorem}[prop]{Theorem}

\theoremstyle{definition}

\theoremstyle{remark}
\newtheorem{remark}[prop]{Remark}

\newcommand{\Mod}[1]{\ (\mathrm{mod}\ #1)}
\newcommand{\C}{\mathbb{C}}
\newcommand{\R}{\mathbb{R}}

\newcommand{\ZZ}{\mathbb{Z}}
\newcommand{\CC}{\mathbb{C}}

\newcommand{\sceq}{\mathrel{\mathop:}=}

\newcommand{\pmat}[4]{\begin{pmatrix} #1&#2\\#3&#4\end{pmatrix}}

\DeclareMathOperator{\tr}{tr}

\DeclareMathOperator{\dist}{dist}
\DeclareMathOperator{\dvol}{dvol}
\DeclareMathOperator{\vol}{vol}

\newcommand*{\defeq}{\stackrel{\text{def}}{=}}
\newcommand*{\mseq}{\stackrel{\text{m}}{=}}

\allowdisplaybreaks

\makeatletter
\@namedef{subjclassname@2020}{%
  \textup{2020} Mathematics Subject Classification}
\makeatother

\usetikzlibrary{patterns}

\begin{document}

\title[]{Explicit spectral gap for Hecke congruence covers of arithmetic Schottky surfaces}

\author[L.\@ Soares]{Louis Soares}
\email{louis.soares@gmx.ch}

\begin{abstract}
Let $\Gamma$ be a Schottky subgroup of $\mathrm{SL}_2(\mathbb{Z})$ and let $X=\Gamma\backslash \mathbb{H}^2$ be the associated hyperbolic surface. Conditional on the generalized Riemann hypothesis for quadratic $L$-functions, we establish a uniform and explicit spectral gap for the Laplacian on the Hecke congruence covers $ X_0(p) = \Gamma_0(p)\backslash \mathbb{H}^2$ of $X$ for ``almost'' all primes $p$, provided the limit set of $\Gamma$ is thick enough.
\end{abstract}
\maketitle

\section{Introduction}

\subsection{Spectral gaps for congruence covers and main result} For all $q\in \mathbb{N}$ we denote by $X(q)$ the principal congruence cover of level $q$ of the modular surface $X = \mathrm{SL}_2(\mathbb{Z})\backslash\mathbb{H}^2$ and we let
$$
\lambda_0(q)=0< \lambda_1(q) \leqslant  \lambda_2(q) \leqslant \dots
$$ 
be the eigenvalues of the Laplace--Beltrami operator on $X(q)$. In \cite{Selberg_3/16}, Selberg famously proved that for all $q$ we have $\lambda_1(q)\geqslant \frac{3}{16}$, and he conjectured that for all $q$ we should have $\lambda_1(q)\geqslant \frac{1}{4}$. This remains one of the fundamental open problems of automorphic forms, although notable progress has been made in \cite{Iwaniec1990,LRS,Iwaniec96,KimShahidi,HKSar}, see also the expository articles of Sarnak \cite{Sar95,Sar05}.

In this paper, we consider congruence covers of quotients $X=\Gamma\backslash \mathbb{H}^2$, where $\Gamma$ is an \textit{infinite-index} subgroup of $ \mathrm{SL}_2(\ZZ)$. Such groups, also called ``thin'' groups, do not come under the purview of the aforementioned papers. When $\Gamma$ is a thin group, the Hausdorff dimension $\delta$ of its limit set is smaller than $1$ and $X=\Gamma\backslash \mathbb{H}^2$ is an \textit{infinite-area} hyperbolic surface. Moreover, the $L^2$-spectrum of the Laplace--Beltrami operator $\Delta_X$ on $X$ is rather sparse, see §\ref{sec:specth} for more details. If $\delta > \frac{1}{2}$, there are only finitely many eigenvalues that all lie within the interval $\left[\delta(1-\delta),\frac{1}{4} \right]$, and the smallest eigenvalue is equal to $\lambda_0=\delta(1-\delta)$. If $\delta\leqslant \frac{1}{2}$ there are no eigenvalues at all. Borthwick's book \cite{Borthwick_book} is a good reference for the spectral theory of infinite-area hyperbolic surfaces. 

We focus on the case $\delta > \frac{1}{2}$ and we define the multiset
$$
\Omega(X) \defeq \left\{ s\in \left(\frac{1}{2},\delta\right] : \text{$\lambda = s(1-s)$ is an $L^2$-eigenvalue of $\Delta_X$} \right\},
$$
where each $s$ is repeated according to the multiplicity of $\lambda = s(1-s)$ as an eigenvalue of $\Delta_X$. When $\Gamma$ is a subgroup of $\mathrm{SL}_2(\mathbb{Z})$ and $q\in \mathbb{N}$, we define the (principal) congruence subgroup of $\Gamma$ of level $q$ as usual by
\begin{equation}\label{defiCongrunceGroup}
\Gamma(q) \defeq \left\{ \text{$\gamma\in \Gamma$ : $\gamma\equiv I$ mod $q$} \right\},
\end{equation}
and we write $X(q) = \Gamma(q)\backslash\mathbb{H}^2$ for the associated covering. 

Building on earlier work of Sarnak--Xue \cite{SarnakXue} for cocompact arithmetic\footnote{We refer to \cite{SarnakXue} for the precise definition of ``arithmetic'' in this context} subgroups, Gamburd \cite{Gamburd1} proved the first analogue of Selberg's $\frac{3}{16}$-theorem in the infinite-area setting:

\begin{theorem}[Gamburd \cite{Gamburd1}]\label{thm:gamburd} For every finitely generated subgroup $\Gamma\subset \mathrm{SL}_2(\mathbb{Z})$ with $\delta > \frac{5}{6}$ and for every large enough prime $p$ we have
\begin{equation}\label{eq:gamburd}
\Omega(X(p)) \cap \left( \frac{5}{6}, \delta \right] \mseq \Omega(X) \cap \left( \frac{5}{6}, \delta \right],
\end{equation}
where for any two multisets $A$ and $B$ we write $A \mseq B$ if and only if the multiplicities of all elements are the same on both sides.
\end{theorem}

Theorem \ref{thm:gamburd} implies that the second  eigenvalue of the Laplace--Beltrami operator on $X(p)$, if existent, satisfies 
$$
\lambda_1(p)\geqslant \min\left\{ \frac{5}{36},\lambda_1(1) \right\}.
$$
Selberg's $\frac{3}{16}$-theorem and Gamburd's $\frac{5}{36}$-theorem have been extended as follows:

\begin{theorem}[Bourgain--Gamburd--Sarnak \cite{BGS}]\label{thm:bgs} For every finitely generated subgroup $\Gamma\subset \mathrm{SL}_2(\mathbb{Z})$ with $\delta > \frac{1}{2}$ there exists $c = c(\Gamma)>0$ such that for all square-free $q\in \mathbb{N}$ we have 
\begin{equation}\label{eq:gamburd}
\Omega(X(q)) \cap \left[ \delta - c, \delta \right] \mseq \Omega(X) \cap \left[ \delta - c, \delta \right].
\end{equation}
\end{theorem}
Analogous statements in terms of resonances when $\delta \leqslant \frac{1}{2}$ have been established by Oh--Winter \cite{OhWinter} and Bourgain--Kontorovich--Magee \cite{BKM}, see also \cite{soares2023uniform}. The disadvantage of Theorem \ref{thm:bgs} (as with the results in \cite{OhWinter,BKM,soares2023uniform}) is that it does not provide an \textit{explicit} gap. Apart from the intrinsic interest of this problem, explicit spectral gaps have many applications to concrete dynamical and arithmetic questions, see for instance \cite{Kon09, BK10, KO12, HK15, Ehr19}. 

Recently, Calder\'{o}n--Magee \cite{caldmagee2023} improved Theorem \ref{thm:gamburd} when $\Gamma$ is an arithmetic \textit{Schottky} group. Schottky groups stand out, among other Fuchsian groups, by their simple geometric construction, which we recall in §\ref{sec:SchottkyGroups}.

\begin{theorem}[Calder\'{o}n--Magee \cite{caldmagee2023}]\label{thm:caldmagee} For every Schottky group $\Gamma\subset \mathrm{SL}_2(\mathbb{Z})$ with $\delta > \frac{4}{5}$ and for every $\eta>0$ there exists a constant $C = C(\Gamma,\eta)>0$ such that for all $q\in \mathbb{N}$ whose prime divisors are all greater than $C$, we have
\begin{equation}\label{eq:caldmagee}
\Omega(X(q)) \cap \left[ \frac{\delta}{6} + \frac{2}{3} + \eta, \delta \right] \mseq \Omega(X) \cap \left[ \frac{\delta}{6} + \frac{2}{3} + \eta, \delta \right].
\end{equation}
\end{theorem}

In this paper we pursue a similar goal as in \cite{caldmagee2023}. We wish to establish explicit spectral gaps that go beyond Gamburd's $\frac{5}{36}$-result. We are interested in ``Hecke" congruence subgroups of $\Gamma$:
$$
\Gamma_0(q) \defeq \left\{  \gamma = \pmat{a}{b}{c}{d}\in \Gamma : c\equiv 0\mod q\right\}.
$$
We write $X_0(q) = \Gamma_0(q) \backslash \mathbb{H}^2$ for the associated cover of $X$. Our main result is the following:

\begin{theorem}[Main theorem]\label{Thm:AverageTheorem} Let $\Gamma\subset \mathrm{SL}_2(\mathbb{Z})$ be a Schottky group with $\delta > \frac{3}{4}$. Assume the generalized Riemann hypothesis for quadratic $L$-functions. Then for any fixed $\eta > 0$ there exists a density one subset $\mathcal{P}$ of primes such that for every $p\in \mathcal{P}$ we have
\begin{equation}\label{eq:soares}
\Omega(X_0(p)) \cap \left[ \frac{5}{6} \delta +\eta, \delta \right] \mseq \Omega(X) \cap \left[ \frac{5}{6} \delta +\eta, \delta \right].
\end{equation}
More precisely, for every $\epsilon >0$, the number of primes $p$ not satisfying \eqref{eq:soares} and not exceeding $x$ is at most $O_\epsilon(x^{1-\frac{3}{\delta}\eta +\epsilon})$.  
\end{theorem}

Clearly, since $ \Gamma(q)\subset \Gamma_0(q)$ any eigenvalue of $X_0(q)$ is also an eigenvalue of $X(q)$. Thus, assuming that $\Gamma$ is a Schottky group, the conclusion of Theorem \ref{thm:caldmagee} also holds for $X_0(p)$. Note, however, that $\delta < 1$ implies
$$
 \frac{5}{6} \delta < \frac{\delta}{6} + \frac{2}{3} < \frac{5}{6},
$$
so our spectral gap \eqref{eq:soares} is larger than in Theorem \ref{thm:caldmagee} which in turn is larger than in Theorem \ref{thm:gamburd}.

By ``quadratic'' $L$-functions we mean the $L$-functions $L(s, \chi_d)$ associated to the Kronecker symbol $\chi_d(\cdot) = \left(  \frac{d}{\cdot}\right).$ The generalized Riemann hypothesis (henceforth abbreviated to GRH) for the Dirichlet character $\chi$ is the assertion that if $s\in \mathbb{C}$ satisfies $L(s,\chi)=0$ and if $s$ is a not a negative integer, then $s$ has real part $\frac{1}{2}.$ In fact, our proof only requires GRH for the characters $\chi_d$ with $d\equiv 0$ or $1 \mod 4$. 

It would be interesting to see if the methods of this work yield interesting and new results for similar families of lattices in $\mathrm{SL}_2(\mathbb{R}).$

\subsection{Thick arithmetic Schottky groups} 
At this point the reader may wonder whether Schottky subgroups of $\mathrm{SL}_2(\mathbb{Z})$ with $\delta > \frac{3}{4}$ actually exist. Since this is not completely obvious we provide some explicit examples. In fact, we can construct a sequence of Schottky groups $(\Gamma_m)_{m\in \mathbb{N}}$ such that $\delta_{m} = \delta(\Gamma_m)\to  1$ as $m\to \infty$ in the the following way:
$$
\Gamma_m \defeq \langle g_1^\pm, \dots, g_m^\pm\rangle \subset \mathrm{SL}_2(\mathbb{Z}), \quad g_k = \pmat{4k}{16k^2-1}{1}{4k}.
$$
It is not hard to verify that $g_k$ maps the exterior of the disk $B_{k} = \{ z\in \mathbb{C} : \vert z+4k\vert < 1  \}$ to the interior of $B_{-k} = \{z\in \mathbb{C} : \vert z-4k\vert < 1  \}$. Clearly, the disks $B_{1}, \dots, B_{m}, B_{-1}, \dots, B_{-m}$ are centered on the real line and have mutually disjoint closures, so $\Gamma_m$ is a Schottky group in the sense of the definition in §\ref{sec:SchottkyGroups}. Moreover, $\mathcal{F}_m = \mathbb{H}^2\smallsetminus (B_{1}\cup \dots \cup B_{m} \cup B_{-1}\cup \dots\cup B_{-m})$ provides a fundamental domain for $ \Gamma_m \backslash\mathbb{H}^2$. To see that $\delta_m \to 1$, we borrow an argument from Gamburd at the end of his paper \cite{Gamburd1}. The base eigenvalue of $\Gamma_m \backslash\mathbb{H}^2$ equals $\lambda_0(\Gamma_m) = \delta_{m} (1-\delta_{m})$. On the other hand, the variational characterization of the base eigenvalue says that
$$
\lambda_0(\Gamma_m) = \inf_{\substack{ u\in L^2(\mathcal{F}_m) \\  \nabla u \in L^2(\mathcal{F}_m) }}  \frac{\int_{\mathcal{F}_m}  \vert \nabla u \vert^2 d\mu}{\int_{\mathcal{F}_m}  u^2 d\mu}, \quad d\mu = \frac{dx dy}{y^2}.
$$
Similarly to \cite{Gamburd1} our fundamental domain $\mathcal{F}_m$ is an exterior of mutually disjoint Euclidean disks of radius one  and centered on the real line. Therefore, we can use suitable test-functions $u$ on $\mathcal{F}_m$ similar to those in \cite{Gamburd1} to show that $\lambda_0(\Gamma_m)  \to 0$. From this we conclude that $\delta_m\to 1$.

\subsection{Outline of proof}
We now sketch the proof of our main Theorem \ref{Thm:AverageTheorem}. The methods in the papers mentioned at beginning of the introduction \cite{Selberg_3/16,Iwaniec1990,LRS,Iwaniec96,KimShahidi,HKSar} do not easily generalize to thin groups and rather new tools are required here. Our proof uses some of the same basic ingredients as in \cite{NaudMagee,caldmagee2023,soares2023improved} which we specialize to our setting. Eigenvalues for the Laplacian on $X$ are also eigenvalues for the Laplacian on any finite-degree cover $X'$, such as $X'=X_0(p)$. This is a direct consequence of the Venkov--Zograf formula \eqref{VZ_ind_formula_2}. We call an eigenvalue for $X_0(p)$ ``new'' if it occurs with greater multiplicity than in $X$. Let $\lambda_p^0$ be the induced representation of the identity on $\Gamma_0(p)$ to $\Gamma$ minus the identity on $\Gamma$:
$$
\lambda_p^0 \defeq \mathrm{Ind}_{\Gamma_0(p)}^\Gamma (\textbf{1}_{\Gamma_0(p)}) \ominus \textbf{1}_{\Gamma}.
$$
New eigenvalues $\lambda = s(1-s)$ correspond to zeros $s$ of the twisted Selberg zeta function $Z_{\Gamma}(s,\lambda_p^0)$ in the interval $\left[ \frac{1}{2},\delta \right]$. Our goal is to estimate the number of these zeros in $\left[ \sigma,\delta \right]$ for any $\frac{1}{2}<\sigma < \delta$. To that effect, we recall the Fredholm determinant identity
\begin{equation}\label{intro:fred}
Z_{\Gamma}(s,\lambda_p^0) = \det(1-\mathcal{L}_{s,\lambda_p^0}),
\end{equation}
where for any representation $\rho$ of $\Gamma$, the operator $\mathcal{L}_{s,\rho}$ is the so-called $\rho$-\textit{twisted} transfer operator, defined in terms of the Schottky data used in the geometric construction of $\Gamma$, see §\ref{sec:twistedTO}. In order to produce explicit estimates, we use the \textit{refined} transfer operators $\mathcal{L}_{\tau,s,\lambda_p^0}$, see §\ref{sec:partRefOp}. This type of operators was introduced by Dyatlov--Zworski \cite{DyZw18} and can be seen as ``accelerated'' versions of the standard transfer operator $\mathcal{L}_{s,\lambda_p^0}$, where the acceleration is governed by a (small) ``resolution'' parameter $\tau > 0$. One of the key observations of \cite{DyZw18} is that 1-eigenfunctions of $\mathcal{L}_{s,\lambda_p^0}$ are also 1-eigenfunctions of $\mathcal{L}_{\tau,s,\lambda_p^0}$. This implies that zeros of \eqref{intro:fred} are also zeros of the refined zeta function
\begin{equation}\label{intro:refinedzeta}
\zeta_{\tau}(s,\lambda_p^0) \defeq \det\left( 1-\mathcal{L}_{\tau,s,\lambda_p^0}^{2} \right).
\end{equation}
The key point is choosing the parameter $\tau$ that yields the best upper bound for our final estimate. Using Jensen's formula from complex analysis, we can estimate the number $N_p(\sigma)$ of zeros of \eqref{intro:refinedzeta} in the interval $[\sigma,\delta]$ in terms of the Hilbert--Schmidt norm $\Vert \mathcal{L}_{\tau,s,\lambda_p^0}\Vert_{\mathrm{HS}}$ with $s\approx \sigma.$ Estimating this norm for an individual $p$ seems quite difficult. However, it turns out that we can non-trivially estimate the sum
\begin{equation}\label{outline:hs_sum}
\sum_{\substack{ x/2 <p\leqslant  x \\ \text{$p$ prime} }} \Vert \mathcal{L}_{\tau,s,\lambda_p^0}\Vert_{\mathrm{HS}}^2,
\end{equation}
which is the main novelty in this paper. Thanks to an explicit formula for the Hilbert--Schmidt norm (Lemma \ref{lem:HSnorm}), the task of estimating \eqref{outline:hs_sum} reduces to estimating sums of characters of the representation $\lambda_p^0$
\begin{equation}\label{outline:trace_sum}
\sum_{\substack{ x/2 <p\leqslant  x \\ \text{$p$ prime} }} \mathrm{tr}( \lambda_p^0(\gamma)  )
\end{equation}
for fixed $\gamma\in \Gamma.$ We prove that unless $\gamma\in \Gamma$ equals $\pm I$ modulo $p$, then
$$
\mathrm{tr}( \lambda_p^0(\gamma)  ) = \left( \frac{\tr(\gamma)^2-4}{p}\right),
$$
where $\left( \frac{\cdot}{p}\right)$ is the \textit{Kronecker symbol} modulo $p$. Hence, we need to understand the asymptotic behaviour of
\begin{equation}\label{intro:sumlegendre}
\sum_{\substack{ x/2 <p\leqslant  x \\ \text{$p$ prime} }} \left( \frac{d}{p}\right)
\end{equation}
as $x\to \infty$ for fixed $d$. It is here where we invoke GRH. If $d$ is an integer with $d = 0,1,2 \mod 4$, then $ \chi_d(n) = \left( \frac{d}{n}\right)$ is a Dirichlet character of conductor at most $4\vert d\vert$. Hence, assuming GRH, we obtain that for all such $d$, the modulus of \eqref{intro:sumlegendre} is bounded from above by $O_\epsilon(x^{1/2+\epsilon} d^\epsilon).$ Inserting this bound into \eqref{outline:hs_sum} and using some rather well-known distortion estimates for Schottky groups, we obtain that for all $\tau \gg x^{-2}$ and $s\approx \sigma$,
\begin{equation}\label{intro:sumHS}
\sum_{\substack{ x/2 <p\leqslant  x \\ \text{$p$ prime} }} \Vert \mathcal{L}_{\tau,s,\lambda_p^0}\Vert_{\mathrm{HS}}^2 \ll  \tau^{2\sigma}\left( \tau^{-\delta} x^2 + x^{\frac{1}{2}+\epsilon} \tau^{-2\delta} \right),
\end{equation}
see Proposition \ref{prop:sum_of_HS_norm} for a more precise statement. Taking $\tau \approx x^{-\frac{3}{2\delta}}$, the right hand side is $O_\epsilon(x^{1 - \frac{3}{\delta}(\sigma - \frac{5}{6}\delta)+\epsilon})$. This means that for a ``typical'' prime $p$ we have
$$\Vert \mathcal{L}_{\tau,s,\lambda_p^0}\Vert_{\mathrm{HS}}^2 = O_\epsilon(p^{- \frac{3}{\delta}(\sigma - \frac{5}{6}\delta)+\epsilon}),$$
which is enough to deduce Theorem \ref{Thm:AverageTheorem}. 

We point out that all the known unconditional bounds for character sums over primes (at least those known to the author) have a rather high dependency on the conductor. In our application, we need to estimate the sum \eqref{intro:sumlegendre} with $d$ as large as $d \approx x^A$ for some absolute constant $A > 0$. Using unconditional bounds only leads to weak (and actually useless) estimates in \eqref{intro:sumHS}.

\subsection{Organization of the paper}
In Section \ref{sec:prelims} we gather the basic definitions and tools needed for our main proof of Theorem \ref{Thm:AverageTheorem}. In particular, we introduce Schottky groups, refined transfer operators, and we recall the relation between eigenvalues and zeros of refined zeta functions. The proof of Theorem \ref{Thm:AverageTheorem} is then given in Section \ref{sec:proof}.

\subsection{Notation}
We write $f(x)\ll g(x)$ or $ f(x) = O(g(x)) $ interchangeably to mean that there is a constant $C>0$ such that $\vert f(x)\vert \leqslant  C \vert g(x)\vert$ for all $\vert x\vert \geqslant C$. We write $f(x)\ll_y g(x)$ or $ f(x) = O_y(g(x)) $ to mean that $C$ depends on $y$. We write $C=C(y)$ to emphasize that $C$ depends on $y$. In this paper, all the implied constants are allowed to depend on the Schottky group $\Gamma$, which we assume to be fixed throughout. We write $s=\sigma+it\in \mathbb{C}$ to mean that $\sigma$ and $t$ are the real and imaginary parts of $s$ respectively. Given a finite set $S$, we denote its cardinality by $\vert S\vert.$ We use $p\sim x$ is a shorthand for $x/2<p\leqslant  x.$

\subsection{Acknowledgments}
I thank Irving Calder\'{o}n and Michael Magee for several helpful comments. I am also grateful to Irving Calder\'{o}n for suggesting a more compact formulation and a more elegant proof of Lemma \ref{lem:inducedChar}, which helped improve the presentation.

\section{Preliminaries}\label{sec:prelims}

\subsection{Hyperbolic geometry} Let us recall some basic facts about hyperbolic surfaces, referring the reader to Borthwick's book \cite{Borthwick_book} for a comprehensive discussion. One of the standard models for the hyperbolic plane is the Poincar\'{e} half-plane
$$
\mathbb{H}^2=\{ x+iy\in \C\ :\ y>0\} 
$$
endowed with its standard metric of constant curvature $-1$,
$$
ds^2=\frac{dx^2+dy^2}{y^2}.
$$ 
The group of orientation-preserving isometries of $(\mathbb{H}^2, ds)$ is isomorphic to $\mathrm{PSL}_2(\R)$. It acts on the extended complex plane $\overline{\C} = \C \cup \{ \infty\}$ (and hence also on $\mathbb{H}^2$) by M\"{o}bius transformations
$$
\gamma=\pmat{a}{b}{c}{d}\in \mathrm{PSL}_2(\R),\quad z\in\overline{\C} \Longrightarrow  \gamma(z) = \frac{az+b}{cz+d}.
$$
An element $\gamma\in \mathrm{PSL}_2(\R)$ is either
\begin{itemize}
\item \textit{hyperbolic} if $\vert \mathrm{tr}(\gamma)\vert >2$, which implies that $\gamma$ has two distinct fixed points on the boundary $\partial \mathbb{H}^2$,
\item \textit{parabolic} if $\vert \mathrm{tr}(\gamma)\vert =2$, which implies that $\gamma$ has precisely one fixed point on $\partial \mathbb{H}^2$, or
\item \textit{elliptic} if $\vert \mathrm{tr}(\gamma)\vert < 2$, which implies that $\gamma$ has precisely one fixed point in the hyperbolic plane $\mathbb{H}^2$.
\end{itemize}

\subsection{Hyperbolic surfaces and Fuchsian groups}
Every hyperbolic surface $X$ is isometric to a quotient $\Gamma\backslash \mathbb{H}^2$, where $\Gamma$ is a \textit{Fuchsian} group, that is, a discrete subgroup $\Gamma\subset\mathrm{PSL}_2(\R)$. A Fuchsian group $\Gamma$ is called
\begin{itemize}
\item \textit{torsion-free} if it contains no elliptic elements,
\item \textit{non-cofinite} if the quotient $\Gamma\backslash \mathbb{H}^2$ has infinite-area,
\item \textit{non-elementary} if it is generated by more than one element, and
\item \textit{geometrically finite} if it is finitely generated, which is equivalent with $\Gamma\backslash \mathbb{H}^2$ being geometrically and topologically finite.
\end{itemize}

All the Fuchsian groups $\Gamma$ considered in this paper satisfy all the above conditions. The \textit{limit set} $\Lambda$ of $X$, which is defined as the set of accumulation points of all orbits of the action of $\Gamma$ on $\mathbb{H}^2$, is a Cantor-like fractal subset of the boundary $\partial \mathbb{H}^2 \cong \mathbb{R}\cup \{ \infty \}$. Its Hausdorff dimension, denoted by $\delta$, lies strictly between $0$ and $1$.

Furthermore, $\Gamma$ is called \textit{convex cocompact} if it is finitely generated and if it contains neither parabolic nor elliptic elements. This is equivalent with the \textit{convex core} of $X = \Gamma\backslash \mathbb{H}^2$ being compact. By a result of Button \cite{Button}, every infinite-area, convex cocompact hyperbolic surface $X$ can be realized as the quotient of $\mathbb{H}^2$ by a so-called \textit{Schottky group} $\Gamma$, which we will define in §\ref{sec:SchottkyGroups} below, see also \cite[Theorem~15.3]{Borthwick_book}.

We also remark that since we only work with torsion-free Fuchsian groups in this paper, it makes no difference whether we work with $\mathrm{PSL}_2(\R)$ or with $\mathrm{SL}_2(\R)$, so we will henceforth stick to $\mathrm{SL}_2(\R)$.

\subsection{Spectral theory of infinite-area hyperbolic surfaces}\label{sec:specth}
Let us review some aspects of the spectral theory of infinite-area hyperbolic surfaces. We refer the reader to \cite{Borthwick_book} for an in-depth account of the material given here. The $ L^{2} $-spectrum of the Laplace--Beltrami operator $\Delta_X $ on an infinite-area hyperbolic surface $X$ is rather sparse and was described by Lax--Phillips \cite{Lax_Phillips_I} and Patterson \cite{Patterson} as follows:
\begin{itemize}
\item The absolutely continuous spectrum is equal to $ [1/4, \infty) $.

\item The pure point spectrum is finite and contained in the interval $(0,1/4)$. In particular, there are no eigenvalues embedded in the continuous spectrum.

\item If $\delta \leqslant 1/2$ then the pure point spectrum is empty. If $\delta>1/2$ then $\lambda_0(X) = \delta(1-\delta)$ is the smallest eigenvalue.
\end{itemize}
In light of these facts, the resolvent operator
$$
R_X(s) \sceq \big( \Delta_X - s(1-s)\big)^{-1}\colon L^2(X) \to L^2(X)
$$
is defined for all $s\in\CC$ with $ \mathrm{Re}(s)>1/2$ and $s(1-s)$ not being an $L^2$-eigenvalue of $\Delta_X$. From Guillop\'{e}--Zworski \cite{GuiZwor2} we know that the resolvent extends to a meromorphic family
\begin{equation}\label{resolvent_continued}
R_X(s) \colon C_{c}^{\infty}(X) \to C^{\infty}(X)
\end{equation}
on $\CC$ with poles of finite rank. The poles of $R_X(s)$ are called the \textit{resonances} of $ X $ and the multiplicity of a resonance $\zeta$ is the rank of the residue operator of $R_X(s)$ at $s=\zeta$. We denote by $\mathcal{R}(X)$ the set multiset of resonances of $X$ repeated according to multiplicities. Resonances are contained in the half-plane $\mathrm{Re}(s)\leqslant  \delta$, with no resonances on the vertical line $\mathrm{Re}(s)=\delta$ other than a simple resonance at $s=\delta$. Note that resonances $s$ on the half-plane $\mathrm{Re}(s)>\frac{1}{2}$ correspond to eigenvalues $\lambda = s(1-s)$ of $\Delta_X$. In other words, the set $\Omega(X)$ defined in the introduction can be re-expressed as
$$
\Omega(X) \mseq \mathcal{R}(X) \cap \{ \mathrm{Re}(s) > \frac{1}{2} \}.
$$ 
In particular, if $\delta \leqslant  \frac{1}{2}$, then the set $\Omega(X) $ is empty. 

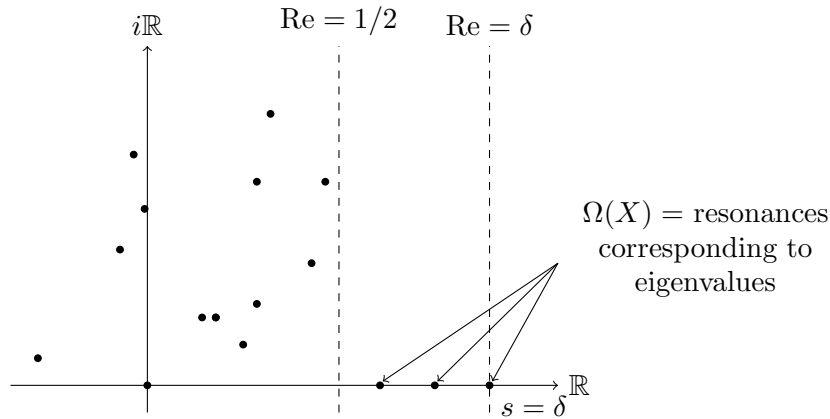
\begin{figure}[h]
\centering
\captionsetup{justification=centering}
\begin{tikzpicture}[xscale=1.8, yscale=1.8]
    \draw [->] (-1,0) -- (3,0) node [right, font=\small]  {$  \mathbb{R} $};
    \draw [->] (0,-0.2) -- (0,2.5) node [above, font=\small] {$ i \mathbb{R} $};
    \draw[dashed] (2.5,-0.2) -- (2.5,2.5) node [above, font=\small] {$ \mathrm{Re}= \delta $};
	\draw[dashed] (1.4,-0.2) -- (1.4,2.5) node [above, font=\small] {$ \mathrm{Re} = 1/2 $};
    \filldraw (2.5,0) circle (0.7pt) node[below right, font=\small] {$s=\delta$};
    \filldraw (1.7,0) circle (0.7pt);
    \filldraw (2.1,0) circle (0.7pt);
	\filldraw (0,0) circle (0.7pt);
	\filldraw (0.8,0.6) circle (0.7pt);
	\filldraw (0.4,0.5) circle (0.7pt);
	\filldraw (0.7,0.3) circle (0.7pt);
	\filldraw (-0.2,1) circle (0.7pt);
	\filldraw (-0.8,0.2) circle (0.7pt);
	\filldraw (0.9,2) circle (0.7pt);
	\filldraw (-0.1,1.7) circle (0.7pt);
	\filldraw (0.8,1.5) circle (0.7pt);
	\filldraw (-0.02,1.3) circle (0.7pt);
	\filldraw (0.5,0.5) circle (0.7pt);
	\filldraw (1.3,1.5) circle (0.7pt);
	\filldraw (1.2,0.9) circle (0.7pt);
	\filldraw (0.5,0.5) circle (0.7pt);
   	\draw (3,1) node[right, font=\small] {\begin{tabular}{c}
    $\Omega(X) =$ resonances \\
    corresponding to \\ 
    eigenvalues
\end{tabular}};
	\draw[->] (3,0.9) -- (2.52,0.03);
   	\draw[->] (3,0.9) -- (1.72,0.03);
   	\draw[->] (3,0.9) -- (2.12,0.03);
\end{tikzpicture}
\caption{Distribution of resonances for infinite-area $ \Gamma\backslash\mathbb{H}^2 $ in the case $ \delta > \frac{1}{2} $}
\label{fig:resonance_infinite_area}
\end{figure}

\subsection{Twisted Selberg zeta function}
Given a finitely generated Fuchsian group $\Gamma<\mathrm{PSL}_2(\R)$, the set of prime periodic geodesics on $X=\Gamma\backslash\mathbb{H}^2$ is bijective to the set $[\Gamma]_{\mathsf{prim}}$ of $\Gamma$-conjugacy classes of primitive hyperbolic elements in $\Gamma$. We denote by $\ell(\gamma)$ the length of the geodesic corresponding to the conjugacy class $[\gamma] \in [\Gamma]_{\mathsf{prim}}$.

The Selberg zeta function is defined for $\mathrm{Re}(s) > \delta$ by the infinite product
\begin{equation}\label{selbergZetaDefi}
Z_{\Gamma}(s)\defeq\prod_{k=0}^\infty \prod_{[\gamma] \in [\Gamma]_{\mathsf{prim}}}\left( 1 - e^{-(s+k)\ell(\gamma)}\right),
\end{equation}
and it has a meromorphic continuation to $s\in \mathbb{C}.$ By Patterson--Perry \cite{Patt_Perry} the zero set of $Z_\Gamma(s)$ consists of the so-called ``topological'' zeros at $s= -k$ for $k\in \mathbb{N}_0$, and the set of resonances, repeated according to multiplicity. Therefore, any problem about resonances and eigenvalues can be rephrased as a question about the distribution of the zeros of the Selberg zeta function.

Given a finite-dimensional, unitary representation $(\rho,V)$ of $\Gamma$, we define the \textit{twisted} Selberg zeta function by
\begin{equation}\label{def_szf_twisted}
Z_{\Gamma}(s,\rho)\defeq\prod_{k=0}^\infty \prod_{[\gamma] \in [\Gamma]_{\mathsf{prim}}}   \det\nolimits_V\left( I_V - \rho(\gamma)e^{-(s+k)\ell(\gamma)}\right).
\end{equation}
Clearly, if $\rho = \textbf{1}_\CC$ is the trivial, one-dimensional representation of $\Gamma$, then \eqref{def_szf_twisted} reduces to classical Selberg zeta function \eqref{selbergZetaDefi}. Furthermore, the product definition \eqref{def_szf_twisted} implies that we have the factorization
\begin{equation}\label{decompFormulaZeta}
Z_\Gamma(s,\rho_1 \oplus \rho_2) = Z_\Gamma(s,\rho_1)Z_\Gamma(s,\rho_2),
\end{equation}
where $\rho_1 \oplus \rho_2$ denotes the orthogonal direct sum of $\rho_1$ and $\rho_2$.

\subsection{Venkov--Zograf induction formula}\label{sec:VZ}
The reason we are interested in twisted Selberg zeta functions is because of the \textit{Venkov--Zograf induction formula} \cite{VenkovZograf,Venkov_book}. It says that if $\Gamma'$ is a finite-index subgroup of $\Gamma$, then we have 
\begin{equation}\label{VZ_ind_formula}
Z_{\Gamma'}(s) = Z_\Gamma(s,\lambda_{\Gamma/\Gamma'}).
\end{equation}
where
\begin{equation}
\lambda_{\Gamma/\Gamma'} \defeq \mathrm{Ind}_{\Gamma'}^{\Gamma}(\textbf{1}_{\Gamma'})
\end{equation}
is the \textit{induced representation} of the trivial one-dimensional representation $\textbf{1}_{\Gamma'}$ of $\Gamma'$ to the larger group $\Gamma.$ See also the more recent paper \cite{FP_szf} for a proof of this formula based on the Frobenius character formula.

Let $g_1, \dots, g_n$ be a full set of representatives in $\Gamma$ of the left cosets in $\Gamma/\Gamma'$, where $n=[\Gamma:\Gamma']$ is the index of $\Gamma'$ in $\Gamma$. Then the induced representation can be thought of as acting on the space
\begin{equation}\label{reprspaceind}
V_{\Gamma/\Gamma'} \defeq \mathrm{span}_\mathbb{C} \{ g_1, \dots, g_n \} = \left\{ \sum_{i=1}^n \alpha_i g_i : \alpha_1, \dots, \alpha_n\in \mathbb{C} \right\}.
\end{equation} 
By definition, for each $\gamma\in \Gamma$ and for each $i\in [n]$ there exists $\sigma(i)\in [n]$ and $\tilde{\gamma}\in \Gamma'$ such that $ \gamma g_i = g_{\sigma(i)} \tilde{\gamma}. $ The action of $\lambda_{\Gamma/\Gamma'}$ is then given by
$$
\lambda_{\Gamma/\Gamma'}(\gamma)\left( \sum_{i=1}^n \alpha_i g_i \right) = \sum_{i=1}^n \alpha_i g_{\sigma(i)}.
$$
In fact, $\sigma\in S_n$ is a permutation of $[n]$ and with respect to the basis $\{ g_1, \dots, g_n\}$, $\lambda_{\Gamma/\Gamma'}(\gamma)$ acts on $V_{\Gamma/\Gamma'}$ by the permutation matrix associated to $\sigma.$ Moreover, the induced representation splits as an orthogonal direct sum
$$
\lambda_{\Gamma/\Gamma'} = \textbf{1}_\Gamma  \oplus \lambda_{\Gamma/\Gamma'}^0,
$$
where $\lambda_{\Gamma/\Gamma'}^0$ is an $(n-1)$-dimensional representation acting on the subspace
\begin{equation}\label{reprspaceind0}
V_{\Gamma/\Gamma'}^0 \defeq \left\{ \sum_{i=1}^n \alpha_i g_i \in V_{\Gamma/\Gamma'} :  \sum_{i=1}^n \alpha_i = 0  \right\}.
\end{equation}
Thanks to \eqref{decompFormulaZeta} we have
\begin{equation}\label{VZ_ind_formula_2}
Z_{\Gamma'}(s) = Z_\Gamma(s) Z_\Gamma(s,\lambda_{\Gamma/\Gamma'}^0).
\end{equation}
We conclude that ``new'' resonances for $X' = \Gamma'\backslash\mathbb{H}^2$ (that is, resonances which have greater multiplicity in $X'$ than in $X$) appear as zeros of $Z_\Gamma(s,\lambda_{\Gamma/\Gamma'}^0).$ In particular, if $\lambda$ is a ``new'' eigenvalue for $X'$, then we have $\lambda = s(1-s)$ for some $s\in [\frac{1}{2},\delta]$ with $Z_\Gamma(s,\lambda_{\Gamma/\Gamma'}^0)=0.$

\subsection{Schottky groups}\label{sec:SchottkyGroups}
Let us now recall the definition of Schottky groups.
\begin{itemize}
\item Define the alphabet $ \mathcal{A} = \{ 1, \dots, 2m\} $ and for each $ a\in \mathcal{A} $ define 
$$
\overline{a} \defeq \begin{cases} a+r &\text{ if } a \in \{ 1,\dots, m\} \\ a-m &\text{ if } a \in \{ m+1,\dots, 2m\}   \end{cases}
$$

\item Fix open disks $D_1, \dots, D_{2m} \subset \mathbb{C}$ centered on the real line with mutually disjoint closures.

\item Fix isometries $ \gamma_1, \dots \gamma_{2m}\in \mathrm{SL}_2(\mathbb{R}) $ such that for all $ a\in \mathcal{A} $
$$
\text{$\gamma_a ( \overline{\mathbb{C}}\smallsetminus D_{\overline{a}} ) = D_a $ and $\gamma_{\overline{a}} = \gamma_a^{-1}.$}
$$
(In the notation of \cite[§15]{Borthwick_book} we have $m=r$ and $\gamma_a = S_a^{-1}$.)

\item Let $ \Gamma \subset \mathrm{SL}_2(\R) $ be the group generated by the elements $ \gamma_1, \dots \gamma_{2m}$. This is a free group on $ m $ generators, see for instance \cite[Lemma~15.2]{Borthwick_book}.
\end{itemize}

\begin{figure}[H]
\centering
\captionsetup{justification=centering}
\begin{tikzpicture}[xscale=1.5, yscale=1.5]
\draw (-4,0) -- (4,0) node [right, font=\small]  {$  \partial \mathbb{H}^2 $};
\draw (0,1) node [above, font=\small]  {$ \mathbb{H}^2 $};
\draw (-3,0) circle [radius = 0.4];
\draw (-1.3,0) circle [radius = 0.8];
\draw (-0.1,0) circle [radius = 0.3];
\draw (1,0) circle [radius = 0.5];
\draw (2.3,0) circle [radius = 0.4];
\draw (3.5,0) circle [radius = 0.2];
\draw (-3,-0.4) node[below,font=\small] {$ D_{1} $};
\draw (-1.3,-0.8) node[below,font=\small] {$ D_{4} $};
\draw (-0.1,-0.3) node[below,font=\small] {$ D_{2} $};
\draw (1,-0.5) node[below,font=\small] {$ D_{3} $};
\draw (2.3,-0.4) node[below,font=\small] {$ D_{5} $};
\draw (3.5,-0.2) node[below,font=\small] {$ D_{6} $};
 \draw [arrow, bend left]  (-3,0.4) to (-1.3,0.8);
 \draw [arrow, bend left]  (-0.1,0.3) to (2.3,0.4);
 \draw [arrow, bend left]  (1,0.5) to (3.5,0.2);
 \draw (-2.2,0.9) node[above,font=\small] {$ \gamma_{1} $};
 \draw (1.1,0.7) node[above,font=\small] {$ \gamma_{2} $};
 \draw (2.4,0.7) node[above,font=\small] {$ \gamma_{3} $};
\end{tikzpicture}
\caption{A configuration of Schottky disks and isometries with $ m=3 $}
\end{figure}
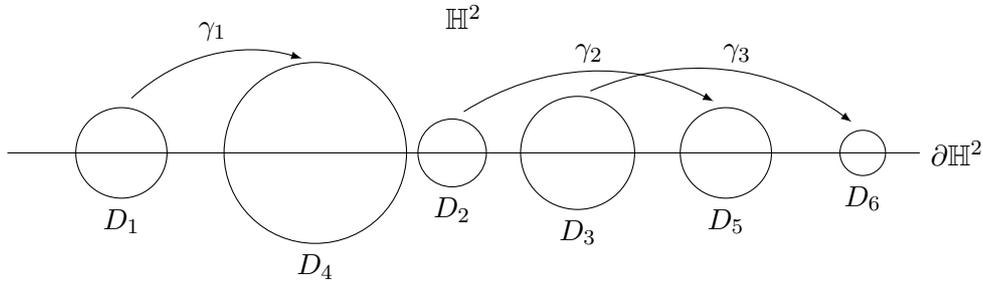

\textit{Throughout the rest of this paper, $\Gamma$ is a non-elementary Schottky group with Schottky data $D_1, \dots, D_{2m}$ and $\gamma_1, \dots, \gamma_{2m}$ as above. This assumption will not be repeated in the sequel.}

\subsection{Combinatorial notation for words}\label{sec:comNotation} 
Let $\Gamma$ be a Schottky group as in §\ref{sec:SchottkyGroups}. We will follow the combinatorial notation of Dyatlov--Zworski \cite{DyZw18} for indexing elements in the free group $\Gamma$.

\begin{itemize}
\setlength\itemsep{1em}
\item A word $\textbf{a}$ in the alphabet $\mathcal{A} =\{ 1,\dots,2m \}$ is a finite string $\textbf{a}=a_1 \dots a_n$ with $a_1, \dots, a_n \in \mathcal{A}$. For technical reasons, we also introduce the empty word $\emptyset$, a string of length zero. 

\item A word $\textbf{a}=a_1 \dots a_n$ is said to be ``reduced'' if $a_{j} \neq \overline{a_{j+1}} $ for all $j=1,\dots, n-1$. For all $n\in\mathbb{N}$ denote by $\mathcal{W}_n$ the set of (reduced) words of length $ n $:
$$
\mathcal{W}_n = \left\{  a_1\cdots a_n : \text{$a_1, \dots, a_n \in \mathcal{A}$ s.t. $a_{j} \neq \overline{a_{j+1}} $ for all $j=1,\dots, n-1$} \right\}.
$$
Moreover, put $ \mathcal{W}_0 = \{ \emptyset \} $ where $ \emptyset $ is the empty word.

\item Let $ \mathcal{W} = \bigsqcup_{n\geqslant  0} \mathcal{W}_n $ be the set of all reduced words and write $ \vert \textbf{a}\vert = n $ if $ \textbf{a}\in \mathcal{W}_n $. In other words, $\vert \textbf{a}\vert$ is the reduced word length of $\textbf{a}$. Given $m\in \mathbb{N}$ let $\mathcal{W}_{\geqslant  m} = \bigsqcup_{n\geqslant  m} \mathcal{W}_n$ the set of all reduced words whose length is at least $m$, and let $ \mathcal{W}^\circ = \mathcal{W}_{\geqslant  1}$ be the set of all non-empty reduced words.

\item Given a word $ \textbf{a} = a_1\cdots a_n \in \mathcal{W}^\circ $ write $ \textbf{a}' = a_1\cdots a_{n-1} \in \mathcal{W} $. Note that $ \mathcal{W} $ is a tree with root $ \emptyset $ and $ \textbf{a}' $ is the parent of $ \textbf{a}. $

\item For $ \textbf{a} = a_1\cdots a_n \in \mathcal{W} $ and $ \textbf{b} = b_1\cdots b_m \in \mathcal{W} $ write $ \textbf{a}\to \textbf{b} $ if either $ \textbf{a}= \emptyset $, or $ \textbf{b}= \emptyset $, or $ a_n \neq \overline{b_1} $. Note that in this case, $ \textbf{a}\textbf{b} \in \mathcal{W} $, that is, if $ \textbf{a}\to \textbf{b} $, then the concatenation $\textbf{a}\textbf{b}$ is also a reduced word.

\item Given a word $\textbf{a} = a_1\cdots a_n$ let $\overline{\textbf{a}} = \overline{a_1}\cdots \overline{a_n}$ be its ``mirror'' word.

\item Write $ \textbf{a}\prec \textbf{b} $ if $ \textbf{a} $ is a prefix $ \textbf{b} $, that is, if $ \textbf{b} = \textbf{ac} $ for some $ \textbf{c}\in \mathcal{W} $.

\item We have the one-to-one correspondence
$$
\textbf{a}= a_1\cdots a_n\in \mathcal{W} \mapsto \gamma_\textbf{a} = \gamma_{a_1}\cdots \gamma_{a_n}\in \Gamma. 
$$
Moreover, we have $ \gamma_{\textbf{ab}} = \gamma_{\textbf{a}} \gamma_{\textbf{b}} $, $ \gamma_{\textbf{a}}^{-1} = \gamma_{\overline{\textbf{a}}} $, and $\gamma_\textbf{a} = I $ if and only if $ \textbf{a} = \emptyset. $

\item For $ \textbf{a} = a_1 \cdots a_n \in \mathcal{W}^\circ $ we define the disk
$$
D_\textbf{a}:=\gamma_{\textbf{a}'}(D_{a_n}).
$$
If $ \textbf{a}\prec \textbf{b} $ then $ D_\textbf{a}\subset D_\textbf{b} $. On the other hand, if $ \textbf{a}\nprec\textbf{b} $ and $ \textbf{b}\nprec\textbf{a} $ then $ D_\textbf{a}\cap D_\textbf{b} = \emptyset $. We define the interval 
\begin{equation}\label{defiI_a}
I_\textbf{a} := D_\textbf{a}\cap \mathbb{R}
\end{equation}
and we denote by $ \vert I_\textbf{a}\vert $ its length which is equal to the diameter of $ D_\textbf{a}. $


\item Denote
$$
\text{$D = \bigsqcup_{a\in \mathcal{A}} D_a $ and $ I = \bigsqcup_{a\in \mathcal{A}} I_a$.}
$$

\item In the above notation, the limit set of $\Gamma$ may be re-expressed as follows:
$$
\Lambda = \bigcap_{n\geqslant  1} \bigsqcup_{\textbf{a}\in \mathcal{W}_n} I_\textbf{a} \subset \mathbb{R}.
$$
\end{itemize}

\subsection{Twisted transfer operators}\label{sec:twistedTO}
In what follows, let $V$ be a finite-dimensional complex vector space with hermitian inner product $\langle\cdot,\cdot\rangle_V$ and induced norm $\Vert v\Vert_V = \sqrt{\langle v,v\rangle_V}.$ Let $\rho\colon\Gamma\to\mathrm{U}(V)$ be a unitary representation. Here, ``unitary'' means that for all $\gamma\in \Gamma$ and $ v,w\in V$ we have $\langle\rho(\gamma)v,\rho(\gamma)w\rangle_V = \langle v,w\rangle_V$ and in particular $\Vert \rho(\gamma)v\Vert_V = \Vert v\Vert_V$.

We let $H^2(D,V)$ be the Hilbert space of $V$-valued, square-integrable, holomorphic functions on $D = \bigsqcup_{a\in \mathcal{A}} D_a $:
\begin{equation}\label{defi_bergman}
H^2(\Omega,V) \defeq  \left\{ \text{$f\colon D\to V$ holomorphic} \left\vert \; \left\Vert f \right\Vert < \infty \right.\right\},
\end{equation}
with $L^2$-norm given by
$$
\Vert f\Vert^2 \defeq \int_D \Vert f(z)\Vert_V^2 \dvol(z).
$$
Here ``$\vol$'' denotes the Lebesgue measure on the complex plane. On this space, we define for all $s\in \mathbb{C}$ the \textit{twisted transfer operator}
\begin{equation}\label{op}
\mathcal{L}_{s,\rho}\colon H^2(D,V) \to H^2(D,V)
\end{equation}
by the formula
\begin{equation}\label{transf_defi}
\mathcal{L}_{s,\rho}f(z) \defeq \sum_{\substack{ a \in \mathcal{A} \\ a\to b }}^{2m} \gamma_a'(z)^s \rho(\gamma_a)^{-1} f(\gamma_a (z))\ \mathrm{if}\ z\in D_b.
\end{equation}
Note that the derivative on the right satisfies $\gamma_a'(z)>0$ for all $z\in I_b = D_b\cap \R$, so the complex power $\gamma_a'(z)^s$ is uniquely defined and holomorphic for $z\in D_b$ and $s\in \mathbb{C}.$ More concretely, we define
$$
\gamma_a'(z)^s \defeq \exp( s \mathbb{L}(\gamma_a'(z)) ),
$$
where
\begin{equation}\label{complexLog}
\mathbb{L}(z) = \log \vert z\vert + \mathrm{arg}(z),
\end{equation}
with $\mathrm{arg}\colon \mathbb{C}\smallsetminus (-\infty, 0] \to (-\pi,\pi) $ being the principal value of the argument. 

When $V=\mathbb{C}$ and $\rho = \textbf{1}_\Gamma$ is the trivial, one-dimensional representation, the functional space $H^2(\Omega,V)$ reduces to the classical \textit{Bergman space} $H^2(D)$, and \eqref{transf_defi} reduces to the well-known transfer operator $\mathcal{L}_s = \mathcal{L}_{s, \textbf{1}_\Gamma}$ which can be found for instance in Borthwick's book \cite[Chapter 15]{Borthwick_book}. The operator \eqref{op} is trace class for every $s\in\mathbb{C}$ and its Fredholm determinant is equal to the twisted Selberg zeta function of the Schottky group $\Gamma$, see for instance \cite{JNS}:
\begin{equation}\label{eq:fredholm_identity}
Z_{\Gamma}(s,\rho) = \det(1-\mathcal{L}_{s,\rho}).
\end{equation}
In particular, since $\mathcal{L}_{s,\rho}$ depends holomorphically on $s\in \mathbb{C}$, this identity shows that $Z_{\Gamma}(s,\rho)$ extends to an entire function.

\subsection{Partitions and refined transfer operators}\label{sec:partRefOp}
Now we define \textit{refined} transfer operators which were introduced by Dyatlov--Zworski \cite{DyZw18}. These are generalizations of the standard transfer operator $\mathcal{L}_{s}$. Given a finite subset $Z\subset \mathcal{W}$ we put
\begin{itemize}
\item $Z'=\{ \textbf{a}' : \textbf{a}\in Z \}$ and
\item $\overline{Z} =\{ \overline{\textbf{a}} : \textbf{a}\in Z \}.$
\end{itemize}
For all $s\in \mathbb{C}$ and all finite-dimensional, unitary representations $\rho\colon \Gamma\to\mathrm{U}(V)$ we define the operator
\begin{equation}
\mathcal{L}_{Z,s,\rho}\colon H^2(D,V) \to H^2(D,V)
\end{equation}
by the formula
\begin{equation}\label{generalizedOp}
\mathcal{L}_{Z,s,\rho}f(z) \defeq \sum_{\substack{ \textbf{a}\in (\overline{Z})' \\ \textbf{a}\to b}} \gamma_a'(z)^s \rho(\gamma_a)^{-1} f(\gamma_a (z))\ \mathrm{if}\ z\in D_b.
\end{equation}
Note that $\mathcal{L}_{Z,s,\rho}$ reduces to the standard transfer operator $\mathcal{L}_{s}$ if $Z$ is taken to be $\mathcal{W}_2$, the set of reduced words of length two. 

A finite set $Z\subset \mathcal{W}^\circ$ is called a \textit{partition} if there exists $N\in \mathbb{N}$ such that for every reduced word $\textbf{a}\in \mathcal{W}$ with $\vert \textbf{a}\vert \geqslant  N$, there exists a unique $\textbf{b}\in Z$ such that $\textbf{b}\prec \textbf{a}$. In terms of the limit set, a finite set $Z\in \mathcal{W}^\circ$ is a partition if we have the disjoint union
$$
\Lambda = \bigsqcup_{\textbf{b}\in Z} ( I_\textbf{b} \cap \Lambda ).
$$
Trivial examples of partitions are the sets of reduced words $\mathcal{W}_n$ of length $n\geqslant 2$, in which case we have $\mathcal{L}_{\mathcal{W}_n,s} = \mathcal{L}_s^{n-1}.$

The fundamental fact about partitions is the following result of Dyatlov--Zworski \cite{DyZw18}:

\begin{lemma}\label{lem:partEigen}
Let $Z$ be a finite subset of $\mathcal{W}_{\geqslant  2} = \bigsqcup_{n\geqslant  2} \mathcal{W}_n$. If $Z$ is a partition, then for every $f\in H^2(D)$ the following holds true:
$$
\mathcal{L}_{s,\rho} f = f \Longrightarrow \mathcal{L}_{Z,s,\rho} f = f.
$$
\end{lemma}

In other words, 1-eigenfunctions of $\mathcal{L}_s$ are also 1-eigenfunctions of $\mathcal{L}_{Z,s,\rho}$, provided $Z$ is a partition. When combined with the Fredholm determinant identity \eqref{eq:fredholm_identity}, this implies that if $s\in \mathbb{C}$ is a zero of $Z_{\Gamma}(s,\rho)$, then it also is a zero of the (holomorphic) function $s\mapsto\det(1-\mathcal{L}_{Z,s,\rho})$, provided $Z \subset W_{\geqslant 2}$ is a partition.

The partitions relevant in this paper are defined as follows: for any sufficiently small $\tau > 0$, called the ``resolution parameter'', we put
$$
Z(\tau)\defeq \{ \textbf{a}\in \mathcal{W}^{\circ} : \vert I_\textbf{a}\vert \leqslant  \tau < \vert I_{\textbf{a}'}\vert \}.
$$
The set $Z(\tau)$ is a partition by virtue of the fact that the interval length $\vert I_\textbf{a}\vert$ tends to zero as $\vert \textbf{a}\vert \to \infty$. This in turn follows from the definition of the intervals $I_\textbf{a}$ in \eqref{defiI_a} and from the uniform contraction property in Lemma \ref{lem:basicDistEst} below. Finally, we define the \textit{$\tau$-refined transfer operator} by
\begin{equation}\label{defiRefinedTO}
\mathcal{L}_{\tau,s,\rho} \defeq \mathcal{L}_{Z(\tau),s,\rho}.
\end{equation}
Using \eqref{generalizedOp} $\mathcal{L}_{\tau,s,\rho}$ is explicitly given for every $f\in H^2(D,V)$ and $b\in \mathcal{A}$ by
\begin{equation}\label{defi:refinedTO}
\mathcal{L}_{\tau,s,\rho} f(z) = \sum_{\substack{ \textbf{a}\in Y(\tau) \\ \textbf{a}\to b}} \gamma_a'(z)^s \rho(\gamma_a)^{-1} f(\gamma_a (z))\ \mathrm{if}\ z\in D_b,
\end{equation}
where
\begin{equation}\label{defi:Y}
Y(\tau) \defeq \overline{Z(\tau)}'.
\end{equation}
Note that the operator \eqref{defiRefinedTO} is well-defined if and only if $Y(\tau)\subset W^\circ$, or equivalently, $Z(\tau)\subset W_{\geqslant  2}$. This condition is satisfied whenever the resolution parameter $\tau > 0$ is small enough.

The reason for using this particular family of operators is that we can control the size of $Y(\tau)$ as well as the derivatives $\gamma_\textbf{a}'$ when $\textbf{a}\in Y(\tau)$, see Lemma \ref{lem:YZ} below. This is what allows us to obtain an explicit spectral gap in Theorem \ref{Thm:AverageTheorem}.

\subsection{Some useful bounds for Schottky groups}\label{sec:boundsForDeriv}
We now record some useful estimates for Schottky groups when acting on the hyperbolic plane. Following Magee--Naud \cite{NaudMagee}, we use the following notation: for every $a\in \mathcal{A}$ we pick a point $o_a\in D_a$ and for any $\textbf{a}\in \mathcal{W}^\circ$ we set
$$
o_\textbf{a} \defeq o_a
$$
where $a\in \mathcal{A}$ is chosen such that $\textbf{a}\to a$ and we put
$$
\Upsilon_\textbf{a} \defeq \vert \gamma_\textbf{a}'(o_\textbf{a})\vert.
$$
The following basic estimates are due to Naud \cite{Naud14} and Magee--Naud \cite{NaudMagee}:

\begin{lemma}[Basic distortion estimates]\label{lem:basicDistEst}
The following estimates hold true with implied constants depending only on $\Gamma$:
\begin{enumerate}[(i)]
\item \label{part:uniformContr} Uniform contraction: There are constants $ 0 < \theta_1 < \theta_2 < 1 $ and $ C > 0 $ such that for all $ b\in \mathcal{A} $ and for all $ \textbf{a}\in \mathcal{W} $ with $ \textbf{a}\to b $ and $ z\in D_b $ we have
$$
C^{-1}\theta_1^{\vert \textbf{a}\vert} \leqslant  \vert \gamma_\textbf{a}'(z)\vert \leqslant  C\theta_2^{\vert \textbf{a}\vert}.
$$
\item Bounded distortion 1: For all $ b\in \mathcal{A} $ and for all $ \textbf{a}\in \mathcal{W} $ with $ \textbf{a}\to b $ and $ z_1, z_2\in D_b $ we have
$$
C^{-1} \leqslant  \frac{\vert \gamma_\textbf{a}'(z_1)\vert}{\vert \gamma_\textbf{a}'(z_2)\vert} \leqslant  C.
$$
\item Bounded distortion 2:  There exists a constant $C>0$ such that for all $ b_1,b_2\in \mathcal{A} $, all $z_1\in D_{b_1}$ and all $z_2\in D_{b_2}$, and all $\textbf{a}\in \mathcal{W}^\circ$ with $\textbf{a}\to b_1,b_2$ we have
$$
\left\vert \frac{\gamma'_{\textbf{a}}(z_1)}{\gamma'_{\textbf{a}}(z_2)}\right\vert \leqslant  C.
$$
\item \label{part:ups} For all $b\in \mathcal{A}$ and $z\in D_b$ with $\textbf{a}\to b$ we have $\vert \gamma'_\textbf{a}(z)\vert \asymp \Upsilon_\textbf{a}$.
\item For all $\textbf{a}\in \mathcal{W}^\circ$ we have $\Upsilon_\textbf{a}  \asymp \Upsilon_\textbf{a} $.
\item \label{part:mirrorEst} For all $\textbf{a}\in \mathcal{W}^\circ$ we have $\Upsilon_{\overline{\textbf{a}}} \asymp \Upsilon_\textbf{a} $.
\item \label{part:almostMultip} For all $\textbf{a},\textbf{b}\in \mathcal{W}^\circ$ with $\textbf{a}\to \textbf{b}$ we have $\Upsilon_{\textbf{a}\textbf{b}} \asymp \Upsilon_\textbf{a} \Upsilon_\textbf{b}. $
\item \label{part:boundComplexDeriv} For all $b\in \mathcal{A}$, $\textbf{a}\in \mathcal{W}^\circ$ with $\textbf{a}\to b$, $z\in D_b$, and $s=\sigma+it$ we have 
$$
\vert \gamma_\textbf{a}'(z)^s\vert \ll C^\sigma \Upsilon_\textbf{a}^\sigma e^{C\vert t\vert},
$$
where $C>0$ and the implied constant depend solely on $\Gamma$.
\end{enumerate}
\end{lemma}

The following following estimates concerning the sets $Z(\tau)$ and $Y(\tau)$ are also crucial:

\begin{lemma}[Estimates for $Z(\tau)$ and $Y(\tau)$]\label{lem:YZ}
For all $\tau>0$ small enough the following estimates hold true with implied constants depending only on $\Gamma$:
\begin{enumerate}[(i)]
\item \label{part:XY1} For all $ \textbf{a}\in Z(\tau)$ we have $\Upsilon_\textbf{a} \asymp \tau.$ 
\item \label{part:XY2} For all $ \textbf{a}\in Y(\tau)$ we have $\Upsilon_\textbf{a} \asymp \tau.$
\item \label{part:XY3} $ \vert Y(\tau)\vert \asymp \vert Z(\tau)\vert \asymp \tau^{-\delta}.$
\item \label{part:normlem:YZ} For all $\textbf{a}\in Y(\tau)$ we have
$$
\Vert \gamma_{\textbf{a}}\Vert \asymp \tau^{-1/2},
$$
where $\Vert \cdot\Vert$ is the Frobenius norm
$$
\Vert \pmat{a}{b}{c}{d} \Vert = \sqrt{a^2 + b^2 + c^2 + d^2}. 
$$
\end{enumerate}
\end{lemma}

\begin{proof}
The estimates for $Z(\tau)$ can be found in \cite{Bourgain_Dyatlov}. It is then easy to deduce the same estimates for $Y(\tau)$.  Alternatively, Parts \eqref{part:XY1}-\eqref{part:XY3} can be deduced from the definitions of the sets $Z(\tau)$ and $Y(\tau)$ and Lemma \ref{lem:basicDistEst} above. Let us now prove Part \eqref{part:normlem:YZ} for which we could not find any reference. For technical reasons we may assume that zero is not contained in any of the Schottky disks $(D_b)_{b\in \mathcal{A}}$. Otherwise, we replace the Schottky group $\Gamma$ by a conjugate $g^{-1}\Gamma g$ with some suitable $g\in\mathrm{SL}_2(\mathbb{R})$. Note that this does not affect the statement, since for all $\Vert \gamma_\textbf{a}\Vert$ large enough, we have $\Vert g^{-1} \gamma_{\textbf{a}}  g \Vert \asymp \Vert \gamma_{\textbf{a}} \Vert $ with positive implied constants depending only on $g$ and $\Gamma$. Writing
$$
\gamma_{\textbf{a}} = \pmat{a}{b}{c}{d}, \quad a,b,c,d \in \mathbb{R}, \quad ad-bc= 1
$$
we calculate $x_\textbf{a} = \gamma_{\overline{\textbf{a}}}(\infty) = -d/a $ and
$$
\gamma_{\textbf{a}}'(z) = \frac{1}{(cz + d)^2} = \frac{1}{c^2(z - x_\textbf{a})^2}.
$$
Now fix $ b\in \mathcal{A} $, $ \textbf{a}\in \mathcal{W}^\circ $ with $ \mathbf{a}\to b $, and $z\in D_b$. If we write $\textbf{a}=a_1 \cdots a_n$, then the condition $\mathbf{a}\to b$ is equivalent to $\overline{a}_n \neq b$. Observe also that $x_\textbf{a} = \gamma_{\overline{\textbf{a}}}(\infty) \in D_{\overline{a}_n}$. Since the Schottky disks have mutually disjoint closures this implies that for all $z\in D_b$ the difference $\vert z-x_\textbf{a} \vert$ is bounded from above and below by some positive constants depending only on $\Gamma$. Thus we have
$$
\vert \gamma'_\textbf{a}(z)\vert \asymp  \frac{1}{c^2}.
$$
If we assume further that $\textbf{a}\in Y(\tau)$ then from Lemma \ref{lem:basicDistEst} we obtain 
$$
\tau \asymp  \Upsilon_{\textbf{a}} \asymp \frac{1}{c^2}
$$
and therefore 
$$
\vert c\vert \asymp \tau^{-1/2}.
$$
Now, according to our assumption, $0\notin D$, so both $\gamma_{\textbf{a}}(0)$ and $\gamma_{\overline{\textbf{a}}}(0)$ lie inside $D$. Hence, since $D\subset \mathbb{C}$ is bounded, we can find a constant $C>0$ depending only on $\Gamma$ such that 
$$
C^{-1} < \vert \gamma_{\textbf{a}}(0)\vert, \vert \gamma_{\overline{\textbf{a}}}(0)\vert < C.
$$
But since $\vert \gamma_{\textbf{a}}(0) \vert = \vert \frac{b}{d}\vert $ and $ \vert \gamma_{\overline{\textbf{a}}}(0)\vert = \vert \frac{b}{a}\vert $, this gives
$$
\vert a\vert \asymp \vert b\vert \asymp \vert d\vert
$$
with implied constants depending only on $\Gamma.$ Finally, combining the relation $ad-bc = 1$ with $\vert c\vert \asymp \tau^{-1/2}$ we conclude that
$$
\vert a\vert \asymp \vert b\vert \asymp \vert c\vert \asymp \vert d\vert  \asymp \tau^{-1/2}.
$$
Therefore,
$$
\Vert \gamma\Vert = \sqrt{a^2 + b^2 + c^2 + d^2} \asymp \tau^{-1/2},
$$
as claimed.
\end{proof}

\subsection{Hilbert--Schmidt norm of refined transfer operators}
Given a trace-class operator $A\colon H\to H$ on a separable Hilbert space $H$, the Hilbert--Schmidt norm is defined by
$$
\Vert A\Vert_\mathrm{HS}^2 \defeq \mathrm{tr} \left( A^{\ast} A \right),
$$
where $A^{\ast}$ denotes the adjoint of $A$. The goal of this subsection is to prove the following:

\begin{lemma}[Hilbert--Schmidt norm]\label{lem:HSnorm}
For any finite-dimensional, unitary representation $\rho\colon \Gamma\to \mathrm{U}(V)$, the Hilbert--Schmidt norm of the operator $\mathcal{L}_{\tau,s,\rho} $ is given by the formula
\begin{equation}\label{HSnormexpli}
\Vert \mathcal{L}_{\tau,s,\rho} \Vert_{\mathrm{HS}}^2 = \sum_{b\in \mathcal{A}}\sum_{\substack{ \textbf{a},\textbf{b}\in Y(\tau) \\ \textbf{a},\textbf{b}\to b }} \mathrm{tr} \left( \rho(\gamma_{\textbf{a}}^{-1} \gamma_{\textbf{b}}) \right) \mathcal{I}_{\textbf{a},\textbf{b}}^{(b)},
\end{equation}
where
$$
\mathcal{I}_{\textbf{a},\textbf{b}}^{(b)} = \int_{D_b} \gamma'_{\textbf{a}}(z)^s  \overline{\gamma'_{\textbf{b}}(z)^s } B_D(\gamma_{\textbf{a}}(z), \gamma_{\textbf{b}}(z)) \dvol(z)
$$
Here $B_{D}(z,w)$ is the reproducing kernel of the Bergman space $H^2(D).$ Moreover, for all $b\in \mathcal{A}$ and for all $\textbf{a},\textbf{b}\in Y(\tau)$ with $\textbf{a},\textbf{b}\to b$ we have the bound
\begin{equation}\label{Ybound}
\vert \mathcal{I}_{\textbf{a},\textbf{b}}^{(b)}\vert \ll  (C\tau)^{2\sigma} e^{C \vert t\vert},
\end{equation}
where $C>0$ and the implied constant depend solely on $\Gamma.$
\end{lemma}

\begin{proof}
Similar proofs of formulas for the Hilbert--Schmidt norm for similar operators can be found in \cite[Lemma 4.7]{NaudMagee} and in \cite[Proposition~5.5]{Pohl_Soares}. We now provide an alternative but essentially equivalent argument. We denote by $B_D(z,w)$ the reproducing kernel of the classical Bergman space $H^2(D)$ over $D = \bigsqcup_{b\in \mathcal{A}} D_b,$. This kernel satisfies $\overline{B_{D}(\cdot,w)}\in H^2(D)$ for all $w\in D$ and
$$
\int_{D} B_{D}(z,w) f(w) \dvol(w) = f(z)
$$
for all $f\in H^2(D)$ and all $z\in D$, and is uniquely defined by these two properties. Hence, using the formula \eqref{defi:refinedTO}, $\mathcal{L}_{\tau,s,\rho} $ can be rewritten as the integral operator
$$
\mathcal{L}_{\tau,s,\rho}  f(z) = \int_D K_{\tau,s,\rho}(z,w) f(w)\dvol(w),
$$
where the kernel is given for all $z,w\in D $ by
$$
K_{\tau,s,\rho}(z,w) = \sum_{\substack{ \textbf{a}\in Y(\tau) \\ \textbf{a}\to b }} \gamma'_{\textbf{a}}(z)^s \rho(\gamma_{\textbf{a}})^{-1} B_D(\gamma_{\textbf{a}}(z),w), \quad \text{if $z\in D_b$.}
$$
Note that if the points $z,w\in D$ are fixed, then $K_{\tau,s,\rho}(z,w)$ is an element of the endomorphism ring $\mathrm{End}(V)$ of $V$. The Hilbert--Schmidt norm on $\mathrm{End}(V)$ is defined by 
$$
\Vert A\Vert_2 = \sqrt{\mathrm{tr}_V(A A^\ast)}, \quad A\in \mathrm{End}(V),
$$
where $\mathrm{tr}_V$ is the trace of $V$. We will drop the subscript $V$ from the notation, writing only $\mathrm{tr}_V = \mathrm{tr}$. For all $z\in D_b$ and $w\in D$ the Hilbert--Schmidt norm of $K_{\tau,s,\rho}(z,w)$ (viewed as an element on $\mathrm{End}(V)$) is given by
\begin{align*}
\Vert K_{\tau,s,\rho}(z,w)\Vert_2^2 &= \tr\left( K_{\tau,s,\rho}(z,w) K_{\tau,s,\rho}(z,w)^\ast \right)\\
&= \sum_{\substack{ \textbf{a},\textbf{b}\in Y(\tau) \\ \textbf{a},\textbf{b}\to b }} \mathrm{tr} \left( \rho(\gamma_{\textbf{a}}^{-1} \gamma_{\textbf{b}}) \right) \gamma'_{\textbf{a}}(z)^s  \overline{\gamma'_{\textbf{b}}(z)^s } B_D(\gamma_{\textbf{a}}(z),w)  \overline{B_D(\gamma_{\textbf{b}}(z),w)}.
\end{align*}
Note that for the second equality we used the unitarity of the representation $\rho$ (which implies in particular that $ \rho(\gamma_{\textbf{a}})^\ast \rho( \gamma_{\textbf{b}}) = \rho(\gamma_{\textbf{a}}^{-1} \gamma_{\textbf{b}}) $). The Hilbert--Schmidt norm of $\mathcal{L}_{\tau,s,\rho} $ can now be computed as follows:
\begin{align*}
\Vert \mathcal{L}_{\tau,s,\rho} \Vert_{\mathrm{HS}}^2 &= \int_D \int_D \Vert K_{\tau,s,\rho}(z,w)\Vert_2^2 \dvol(w)\dvol(z)\\
&= \sum_{b\in \mathcal{A}} \int_{D_b} \int_D \Vert K_{\tau,s,\rho}(z,w)\Vert_2^2 \dvol(w)\dvol(z)\\
&= \sum_{b\in \mathcal{A}} \sum_{\substack{ \textbf{a},\textbf{b}\in Y(\tau) \\ \textbf{a},\textbf{b}\to b }} \mathrm{tr} \left( \rho(\gamma_{\textbf{a}}^{-1} \gamma_{\textbf{b}}) \right)  \mathcal{I}_{\textbf{a},\textbf{b}}^{(b)},
\end{align*}
where
\begin{equation}\label{eq:intAB}
\mathcal{I}_{\textbf{a},\textbf{b}}^{(b)} = \int_{D_b} \gamma'_{\textbf{a}}(z)^s  \overline{\gamma'_{\textbf{b}}(z)^s } \left(  \int_D   B_D(\gamma_{\textbf{a}}(z),w)  \overline{B_D(\gamma_{\textbf{b}}(z),w)} \dvol(w) \right)\dvol(z).
\end{equation}
By the defining property of the Bergman kernel, we have
\begin{align*}
\int_D  B_D(\gamma_{\textbf{a}}(z),w)  \overline{B_D(\gamma_{\textbf{b}}(z),w)} \dvol(w) &= \int_D  B_D(\gamma_{\textbf{a}}(z),w)  B_D(w,\gamma_{\textbf{b}}(z)) \dvol(w) \\
&= B_D(\gamma_{\textbf{a}}(z), \gamma_{\textbf{b}}(z)),
\end{align*}
which when inserted into \eqref{eq:intAB} gives
$$
\mathcal{I}_{\textbf{a},\textbf{b}}^{(b)} = \int_{D_b} \gamma'_{\textbf{a}}(z)^s  \overline{\gamma'_{\textbf{b}}(z)^s } B_D(\gamma_{\textbf{a}}(z), \gamma_{\textbf{b}}(z)) \dvol(z),
$$
completing the proof of \eqref{HSnormexpli}. 

Let us now prove the bound in \eqref{Ybound}. Fix $b\in \mathcal{A}$ and words $\textbf{a},\textbf{b}\in Y(\tau)$ such that $\textbf{a},\textbf{b}\to b$. Combining Lemmas \ref{lem:basicDistEst} and \ref{lem:YZ} implies that for all $z\in D_b$
$$
\vert \gamma_{\textbf{a}}'(z)^{s}\vert \ll (C \tau)^{\sigma} e^{C \vert t\vert} 
$$
for some constant $C=C(\Gamma)>0$. Thus, by the triangle inequality
\begin{align*}
\vert \mathcal{I}_{\textbf{a},\textbf{b}}^{(b)}\vert &\leqslant \int_{D_b} \vert \gamma_{\textbf{a}}'(z)^{s}\vert \vert \gamma_{\textbf{b}}'(z)^{s} \vert  \vert B_D(\gamma_{\textbf{a}}(z), \gamma_{\textbf{b}}(z)) \vert\dvol(z)\\
&\ll (C\tau)^{2\sigma} e^{2C \vert t\vert} \sup_{z\in D_b} \vert B_D(\gamma_{\textbf{a}}(z), \gamma_{\textbf{b}}(z)) \vert.
\end{align*}
It remains to prove that
\begin{equation}\label{desiredBergman}
\sup_{z\in D_b} \vert B_D(\gamma_{\textbf{a}}(z), \gamma_{\textbf{b}}(z)) \vert \ll 1.
\end{equation}
Note that $B_D(z,w)$ equals zero unless the points $z$ and $w$ belong to the same Schottky disk $D_b$, in which case $B_D(z,w) = B_{D_b}(z,w)$. Letting $r_b>0$ and $c_b\in \mathbb{R}$ be the radius and the center of $D_b$, respectively, we have the following explicit formula for the Bergman kernel over $D_b$, see for instance \cite[Chapter 1]{BergmanBook}:
$$
B_{D_b}(z,w) = \frac{r_b^2}{\pi^2 \left( r_b^2 - (z-c_b)(\overline{w}-c_b) \right)^2}.
$$
Using this formula, we deduce that
\begin{equation}\label{bergmanBound}
\vert B_{D_b}(z,w)\vert \ll \frac{1}{\dist( z,\partial D_b ) \dist( w,\partial D_b )},
\end{equation}
where $\dist( z,\partial D_b )$ denotes the minimal euclidean distance from $z$ to the boundary $\partial D_b$. From the uniform contraction property in Lemma \ref{lem:basicDistEst} we deduce that for all $\textbf{a}\to b$ with $\textbf{a}\in \mathcal{W}^\circ$ we have $\dist( \gamma_{\textbf{a}}(z),\partial D )\geqslant  c$ for some constant $c = c(\Gamma)>0$. Inserting this into \eqref{bergmanBound} we obtain the desired bound \eqref{desiredBergman}. This completes the proof.
\end{proof}

\subsection{Refined zeta function and pointwise estimate} 
We now define the \textit{refined zeta function} as the Fredholm determinant
$$
\zeta_{\tau}(s,\rho) \defeq \det\left( 1-\mathcal{L}_{\tau,s,\rho}^{2} \right),
$$
which will be crucial in the next section. In particular, we will need the following:

\begin{lemma}[Pointwise estimate for $\zeta_{\tau}(s,\rho)$]\label{lem:pointwiseEst} For all $\tau > 0$ sufficiently small and $s\in \mathbb{C}$ with $\sigma = \mathrm{Re}(s) > \delta $,
\begin{equation*}
- \log \vert \zeta_{\tau}(s,\rho)\vert \leqslant  \dim(\rho)  \frac{(C \tau)^{2(\sigma-\delta)}}{1-(C \tau)^{2(\delta-\sigma)}},
\end{equation*}
where $C>0$ depends only on $\Gamma$ and $\dim(\rho)$ is the dimension of $\rho.$
\end{lemma}

\begin{proof}
For every separable Hilbert space $H$ and for every trace class operator $A\colon H\to H$ with $\Vert A\Vert_H < 1$, we have the absolutely convergent series expansion
\begin{equation}
\det(1-A) = \exp\left( -\sum_{k=1}^\infty \frac{1}{k} \mathrm{tr}(A^k) \right),
\end{equation}
see for instance \cite{Gohberg_Goldberg_Krupnik}. Taking absolute values and logarithms on both sides yields
\begin{equation}
- \log \vert \det(1-A)\vert \leqslant   \sum_{k=1}^\infty \frac{1}{k} \vert \mathrm{Re} (\mathrm{tr}(A^k)) \vert \leqslant  \sum_{k=1}^\infty \frac{1}{k} \vert \mathrm{tr}(A^k) \vert.
\end{equation}
Applying this to $A=\mathcal{L}_{\tau,s,\rho}^{2}$ with $\sigma = \mathrm{Re}(s)>\delta$ gives
\begin{equation}\label{ineq_exp}
- \log \vert \zeta_{\tau}(s,\rho)\vert \leqslant  \sum_{k=1}^\infty \frac{1}{k} \vert \mathrm{tr}(\mathcal{L}_{\tau,s,\rho}^{2k})\vert.
\end{equation}
From the proof of Proposition 4.8 in Magee--Naud \cite{NaudMagee}, the traces on the right are bounded by
\begin{equation*}
\vert \mathrm{tr}( \mathcal{L}_{\tau,s,\rho}^{2k} ) \vert \leqslant  \dim(\rho) (C \tau)^{2k \sigma} \vert Z(\tau)\vert^{2k},
\end{equation*}
where $C>0$ depends only on $\Gamma.$ By Lemma \ref{lem:YZ} we also have
$$
\vert Z(\tau)\vert\ll \tau^{-\delta}.
$$
Combining the previous two estimates we obtain (possibly with a larger constant $C$)
\begin{equation*}
\vert \mathrm{tr}( \mathcal{L}_{\tau,s,\rho}^{2k} ) \vert \leqslant  \dim(\rho) (C \tau)^{2k (\sigma-\delta)}.
\end{equation*}
Returning to \eqref{ineq_exp} and using the geometric series formula we obtain for all $\tau>0$ small enough,
$$
- \log \vert \zeta_{\tau}(s,\rho)\vert \leqslant  \dim(\rho) \sum_{k=1}^\infty (C \tau)^{2k (\sigma-\delta)} =  \dim(\rho)  \frac{(C \tau)^{2(\sigma-\delta)}}{1-(C \tau)^{2(\delta-\sigma)}},
$$
as claimed.
\end{proof}

\section{Proof of Theorem \ref{Thm:AverageTheorem}}\label{sec:proof}

The goal of this section is to prove our main Theorem \ref{Thm:AverageTheorem}.

\subsection{Reducing the proof to counting zeros}\label{sec:setup}
We say that $\lambda$ is a ``new'' eigenvalue for the Laplacian on $X_0(p)=\Gamma_0(p)\backslash\mathbb{H}^2$ if it occurs with greater multiplicity than in $X=\Gamma\backslash\mathbb{H}^2$ and we define
$$
\Omega^{\mathrm{new}}(X_0(p)) \defeq \left\{ s\in \left[\frac{1}{2},\delta\right] : \text{$\lambda = s(1-s)$ is a new eigenvalue for $X_0(p)$} \right\}.
$$
We denote by $N_p(\sigma)$ the number of new eigenvalues $\lambda = s(1-s)$ with $s\geqslant  \sigma$, or equivalently,
$$
N_p(\sigma) \defeq  \# \Omega^{\mathrm{new}}(X_0(p)) \cap [\sigma,\delta].
$$
This section is actually devoted to prove the following theorem from which Theorem \ref{Thm:AverageTheorem} follows directly:

\begin{theorem}[Main theorem, elaborated]\label{Thm:AverageTheorem_recalled} Let $\Gamma\subset \mathrm{SL}_2(\mathbb{Z})$ be a Schottky group with $\delta > \frac{3}{4}$. Assume GRH for quadratic $L$-functions. Then, for all $x$ sufficiently large depending on $\Gamma$ and for all $\epsilon >0$ we have
\begin{equation}\label{sumOverPrimesOfNumberOfZeros}
\sum_{\substack{ p \leqslant  x \\ \text{$p$ prime} }} N_p(\sigma) \leqslant C(\epsilon,\Gamma) x^{ 1-\frac{3}{\delta}( \sigma-\frac{5}{6}\delta )+\epsilon}
\end{equation}
\end{theorem}

It is easy to see that this implies Theorem \ref{Thm:AverageTheorem}. Fix some $\eta >0$. The bound \eqref{sumOverPrimesOfNumberOfZeros} shows that the number of primes $p$ for which $N_p(\frac{5}{6}\delta + \eta) \geqslant  1$ not exceeding $x$ is at most $O_\epsilon(x^{1-\frac{3}{\delta}\eta +\epsilon})$. On the other hand, by the prime number theorem there are roughly $ \frac{x}{\log x}$ primes below $x$, so the number of primes $p$ for which $N_p(\frac{5}{6}\delta + \eta) = 0$ has relative density one.

Let us now turn to the proof of Theorem \ref{Thm:AverageTheorem_recalled}. We use a dyadic decomposition to re-express the sum in \eqref{sumOverPrimesOfNumberOfZeros} as
\begin{equation}\label{dyadic_decomp}
\sum_{\substack{p\leqslant  x \\ \text{$p$ is prime}}} N_p(\sigma) = \sum_{\nu \in \mathbb{N}} S(\frac{x}{2^\nu}, \sigma)
\end{equation}
with
\begin{equation}\label{dyadic_sum}
S(x , \sigma) \defeq \sum_{\substack{p\sim x \\ \text{$p$ is prime}}} N_p(\sigma),
\end{equation}
where $p\sim x$ is a shorthand for $x/2<p\leqslant  x.$ For technical reasons (see also Remark \ref{rmk:dyadic} below), it is more convenient to work with the sums $S(x , \sigma)$. We will prove an estimate of the form
$$
S(x,\sigma) \leqslant \widetilde{C}(\epsilon,\Gamma) x^{ 1-\frac{3}{\delta}( \sigma-\frac{5}{6}\delta )+\epsilon}.
$$
Note that the estimate \eqref{sumOverPrimesOfNumberOfZeros} follows directly from this one.

Recall from §\ref{sec:VZ} that the the Selberg zeta function $Z_{\Gamma_0(p)}(s)$ can be written as
\begin{equation}\label{vz_induction_p}
Z_{\Gamma_0(p)} = Z_{\Gamma}(s, \lambda_p),
\end{equation}
where $\lambda_p = \mathrm{Ind}_{\Gamma_0(p)}^\Gamma (\textbf{1}_{\Gamma_0(p)})$ is the induced representation of the identity $\textbf{1}_{\Gamma_0(p)}$ on the subgroup $\Gamma_0(p)$ to the larger group $\Gamma$. This representation decomposes as
\begin{equation}\label{decompp}
\lambda_p = \textbf{1}_\Gamma \oplus \lambda_p^0.
\end{equation}
In view of \eqref{decompFormulaZeta} we have the factorization
$$
Z_{\Gamma_0(p)}(s) = Z_{\Gamma}(s) Z_{\Gamma}(s,\lambda_p^0).
$$
Therefore, new eigenvalues $\lambda$ for $X_0(p)$ are related to zeros $s$ of $Z_{\Gamma}(s,\lambda_p^0)$ by the equation $\lambda = s(1-s).$ So far we have only reformulated the problem in terms of the zeros of the zeta function:
\begin{equation*}
N_{p}(\sigma) =   \#\left\{ s\in [\sigma, \delta]  :  Z_{\Gamma}(s,\lambda_p^0) = 0 \right\}.
\end{equation*}
This reformulation holds true for any finitely generated subgroup $\Gamma\subset\mathrm{SL}_2(\ZZ)$. Now we invoke the transfer operator machinery for Schottky groups in §\ref{sec:partRefOp}. By Lemma \ref{lem:partEigen}, any zero of $Z_{\Gamma}(s,\lambda_p^0)$ is also a zero of
\begin{equation}
\zeta_{\tau}(s,\lambda_p^0) \defeq \det\left( 1-\mathcal{L}_{\tau,s,\lambda_p^0}^2\right),
\end{equation} 
where $\mathcal{L}_{\tau,s,\lambda_p^0}$ is the refined transfer operator defined in \eqref{defi:refinedTO}. It turns out that the dyadic sums $S(x , \sigma)$ can be estimated using the Hilbert--Schmidt norm of this transfer operator:

\begin{prop}[Zero counting]\label{prop:SumOfHilbertSchmidtNorms} For all $\tau >0$ sufficiently small, for all $K > 1$ sufficiently large, and for all $x$ sufficiently large we have
\begin{equation}
S(x,\sigma) \ll K \max_{\substack{ \mathrm{Re}(s)\geqslant  \sigma-\frac{\alpha}{K} \\ \vert \mathrm{Im}(s)\vert \leqslant  \beta K }} \left( \sum_{\substack{ p\sim x \\ \text{$p$ prime}}} \Vert \mathcal{L}_{\tau, s, \lambda_p^0} \Vert_{\mathrm{HS}}^2 \right) +  x^2 \tau^{K}.
\end{equation}
The implied constant as well as the constants $\alpha > 0$ and $\beta > 0$ depend solely on $\Gamma$.
\end{prop}

\begin{proof}
We use essentially the same argument as in \cite{JN,Pohl_Soares}. We exploit Jensen's formula for holomorphic functions, or rather a weaker variant thereof, which we recall now. Let $f$ be an entire function and consider the pair of concentric disks $D_i = D_{\mathbb{C}}(\sigma_0, r_i)$ with $i\in \{1,2\}$ centered at $\sigma_0\in \mathbb{R}$ and with radii $r_2 > r_1 > 0$. Assume that $\sigma_0, r_1, r_2$ are chosen in such a way that
\begin{equation}\label{concentric_jensen}
[\sigma,\delta]\subset \overline{D_1} \subset D_2.
\end{equation}
Define
$$
M_f(\sigma,\delta) \defeq \# \{ s \in \mathbb{C} : \text{$f(s)=0$, $s\in [\sigma,\delta]$} \}.
$$
Then we have
\begin{equation}\label{jensen}
M_f(\sigma,\delta) \leqslant  \frac{1}{\log( r_2/r_1 )} \left(  \int_{0}^{1} \log\vert f( \sigma_0 + r_2 e^{2\pi i \theta}) \vert d\theta - \log\vert f (\sigma_0)\vert\right).
\end{equation}
Applying this to the refined zeta function $f(s)=\zeta_{\tau}(s,\lambda_p^0)$ we obtain
\begin{equation}\label{jensen_p}
N_p(\sigma) \leqslant  \frac{1}{\log( r_2/r_1 )} \left(  \int_{0}^{1} \log\vert \zeta_{\tau} ( \sigma_0 + r_2 e^{2\pi i \theta}, \lambda_p^0) \vert d\theta - \log\vert \zeta_{\tau} (\sigma_0, \lambda_p^0)\vert \right).
\end{equation}
For all $p$ large enough we have $\dim(\lambda_p^0)=p$, see Lemma \ref{lem:inducedChar} below. Thus, if we assume furthermore that $\sigma_0 > \delta$, then the pointwise estimate in Lemma \ref{lem:pointwiseEst} gives
\begin{equation}\label{jensen_2}
N_p(\sigma) \leqslant  \frac{1}{\log( r_2/r_1 )} \left(  \int_{0}^{1} \log\vert \zeta_{\tau} ( \sigma_0 + r_2 e^{2\pi i \theta}, \lambda_p^0) \vert d\theta +  p \frac{(C \tau)^{2(\sigma_0-\delta)}}{1-(C \tau)^{2(\sigma_0-\delta)}} \right).
\end{equation}
Next, using Weyl's estimate
$$
\log  \vert\det(1-A)\vert \leqslant  \Vert A\Vert_1
$$
together with the Cauchy--Schwarz-type bound
$$
\Vert A_1 A_2\Vert_{1} \leqslant  \Vert A_1\Vert_{\mathrm{HS}}\Vert A_2\Vert_{\mathrm{HS}},
$$
yields
\begin{equation}\label{HS_p_bound}
\log\vert \zeta_{\tau} (s, \lambda_p^0) \vert \leqslant  \Vert \mathcal{L}_{\tau, s, \lambda_p^0}^2\Vert_{1} \leqslant  \Vert \mathcal{L}_{\tau, s, \lambda_p^0}\Vert_{\mathrm{HS}}^2.
\end{equation}
Inserting this into \eqref{jensen_2} gives 
\begin{equation}
N_p(\sigma) \leqslant \frac{1}{\log( r_2/r_1 )} \left(  \int_{0}^{1} \Vert \mathcal{L}_{\tau, \sigma_0 + r_2 e^{2\pi i \theta}, \lambda_p^0} \Vert_{\mathrm{HS}}^2 d\theta +  p \frac{(C \tau)^{2(\sigma_0-\delta)}}{1-(C \tau)^{2(\sigma_0 - \delta)}} \right).
\end{equation}
Summing this inequality over all the primes in $(\frac{x}{2},x]$ with $x$ large enough yields
\begin{equation}\label{first_bound_S}
S(x,\sigma) \leqslant  \frac{1}{\log( r_2/r_1 )} \left( \int_{0}^{1}  \sum_{\substack{ p\sim x \\ \text{$p$ prime}}} \Vert \mathcal{L}_{\tau, \sigma_0 + r_2 e^{2\pi i \theta}, \lambda_p^0} \Vert_{\mathrm{HS}}^2  d\theta +  x^2 \frac{(C \tau)^{2(\sigma_0-\delta)}}{1-(C \tau)^{2(\sigma_0 - \delta)}} \right).
\end{equation}
Let us now choose appropriate parameters $\sigma_0, r_1, r_2.$ For $K>1$, we put
$$
\text{$ \sigma_0 = \delta + K $, $r_1 = \sqrt{(\sigma_0 - \sigma)^2 + 1 }$ and $r_2 = r_1 + 1/K$}.
$$
One can verify that these choices ensure that the inclusions in \eqref{concentric_jensen} hold true. Furthermore, for $K > 1$ large, the following estimates hold true with some absolute implied constants:
\begin{enumerate}[(i)]
\item $r_1 \asymp r_2 \asymp \sigma_0 - \sigma \asymp K,$
\item $ \sqrt{1 + \frac{1}{(\sigma_0 - \sigma)^2} } = 1+O(\frac{1}{K^2}),$
\item $ r_1 = \sqrt{(\sigma_0 - \sigma)^2 + 1 } = (\sigma_0 - \sigma) \sqrt{1 + \frac{1}{(\sigma_0 - \sigma)^2} } = (\sigma_0 - \sigma) + O( \frac{1}{K} ), $ and
\item $ r_2 = \sigma_0 - \sigma + O( \frac{1}{K} ). $
\end{enumerate}
These estimates imply that for all $s=\sigma_0 + r_2 e^{2\pi i \theta}$ with $\theta \in [0,1]$ we have
\begin{equation*}
\text{$\mathrm{Re}(s)\geqslant  \sigma_0 - r_2 \geqslant   \sigma - O(\frac{1}{K})$ and $\vert \mathrm{Im}(s)\vert \leqslant  \sigma_0 + r_2 =  O(K)$.}
\end{equation*}
Therefore, returning to \eqref{first_bound_S}, if $\tau > 0$ is sufficiently small (in terms of $\Gamma$), we obtain
\begin{equation*}
S(x,\sigma) \ll K \max_{\substack{ \mathrm{Re}(s)\geqslant  \sigma-O(\frac{1}{K}) \\ \vert \mathrm{Im}(s)\vert \leqslant  O(K) }} \left( \sum_{\substack{ p\sim x \\ \text{$p$ prime}}} \Vert \mathcal{L}_{\tau, s, \lambda_p^0} \Vert_{\mathrm{HS}}^2 \right) +  x^2 \tau^{K},
\end{equation*}
with all implied constants independent of $x,\tau, K$, as claimed. This establishes Proposition \ref{prop:SumOfHilbertSchmidtNorms}.
\end{proof}

\subsection{The main number-theoretic bound}
Recall that for any subgroup $\Gamma\subset\mathrm{SL}_2(\ZZ)$ we let $\lambda_p$ be the induced representation $\mathrm{Ind}_{\Gamma_0(p)}^\Gamma (\textbf{1}_{\Gamma_0(p)})$ and we define $\lambda_p^0 = \lambda_p \ominus \textbf{1}_\Gamma$. Moreover, we endow the space of $2\times 2$ real matrices with the Frobenius norm
$$
\Vert \pmat{a}{b}{c}{d} \Vert = \sqrt{a^2 + b^2 + c^2 + d^2},
$$
and we write $I=\pmat{1}{0}{0}{1}$ for the identity.

The aim of this subsection is to prove the following:

\begin{prop}[Main number-theoretic bound]\label{prop:inducedSum} Let $\Gamma$ be a finitely generated subgroup of $\mathrm{SL}_2(\ZZ)$. Assume GRH for quadratic $L$-functions. Then, for all $x$ large enough (in terms of $\Gamma$) and for every hyperbolic element $\gamma\in \Gamma$ with
$$
\Vert \gamma\Vert < \frac{1}{20} x^2
$$
we have
\begin{equation}\label{eq:absumm}
\sum_{\substack{ p\sim x \\ \text{$p$ prime} }} \log(p) \mathrm{tr}(\lambda_p^0(\gamma))  =  O(x^{\frac{1}{2}} \log( x )^2)
\end{equation}
with some absolute implied constant.
\end{prop}

\begin{remark}\label{rmk:dyadic}
We remark that in \eqref{eq:absumm} it is not possible to replace $p\sim x$ by $p\leqslant  x$. This is why we need a dyadic decomposition in \eqref{dyadic_decomp}. 
\end{remark}

We recall that the \textit{Legendre symbol} is defined for all integers $a$ and all odd primes $p$ by
$$
\left(  \frac{a}{p}\right) = \begin{cases}
1 &\text{ if $a=x^2 \mod p$ for some $x\in \mathbb{F}_p \smallsetminus \{ 0 \} $}\\
0 &\text{ if $a=0 \mod p$}\\
-1 &\text{ else.}
\end{cases}
$$
There is a standard way of extending the Legendre symbol to a Dirichlet character in the bottom argument. For $p=2$ we define
$$
\left(  \frac{a}{2}\right) = \begin{cases}
0 &\text{ if $a$ is even}\\
1 &\text{ if $a = \pm 1 \mod 8$}\\
-1 &\text{ if $a = \pm 3 \mod 8$}.\\
\end{cases}
$$
Now we define for all $n\in \mathbb{N}$ the \textit{Kronecker symbol} by
$$
\left( \frac{a}{n}\right) = \left(  \frac{a}{p_1}\right)^{r_1} \cdots \left(  \frac{a}{p_m}\right)^{r_m}, 
$$
where $n=p_1^{r_1} \cdots p_m^{r_m}$ is the prime factorization of $n$. Clearly, if $n=p$ is an odd prime, then the Kronecker symbol is just the Legendre symbol. If either the top or bottom argument is fixed, the Kronecker symbol is a completely multiplicative function in the remaining argument. In fact, it is well known that if $d \equiv 0,1$ or $2 \mod 4$, then $\chi_d(n) = \left( \frac{d}{n}\right)$ is a non-principal Dirichlet character of conductor at most $4\vert d\vert$. 

The crucial number-theoretic ingredient in the proof of Proposition \ref{prop:inducedSum} is the following bound which can be extracted from the classical textbook of Iwaniec--Kowalski \cite{AnalyticNumberTheory}.

\begin{theorem}[Special case of Theorem 5.15 in \cite{AnalyticNumberTheory}]\label{thm:iwanieckowalski} Assume GRH for quadratic $L$-functions. Then for all $d\geqslant 1$ with $d \in \{ 0,1,2\} \mod 4$ we have, as $x\to \infty$,
\begin{equation}\label{eq:iwanieckowalski}
\sum_{\substack{ 2\leqslant  p\leqslant  x \\ \text{$p$ prime}}} \log(p)  \left( \frac{d}{p}\right)  =  O( x^{\frac{1}{2}} \log( d x )^2)
\end{equation}
with some absolute implied constant. 
\end{theorem}

Theorem 5.15 in \cite{AnalyticNumberTheory} actually says that for all non-principal characters $\chi$ (not just the Kronecker symbol $\chi_d$) we have 
\begin{equation}\label{eq:iwanieckowalski_vanm}
\sum_{n\leqslant x} \Lambda(n)\chi(n)=  O( x^{\frac{1}{2}} \log( d x )^2),
\end{equation}
where $\Lambda(n)$ is the \textit{von Mangoldt function}
$$
\Lambda(n) = \begin{cases}
\log(p) &\text{ if $n=p^k$ for some $k\in \mathbb{N}$ and some prime $p$}\\
0 &\text{ else}.
\end{cases}
$$
It is easy to deduce \eqref{eq:iwanieckowalski} from \eqref{eq:iwanieckowalski_vanm}. Informally, this statement says that the values of $\chi(p)$, when $p$ ranges over the primes (in increasing order) vary extremely randomly. The reason we are specifically interested in the Kronecker symbol will become clear in the next lemma. Before stating it, we define for all $n\in \mathbb{N}$ the reduction modulo $n$ map by
$$
\pi_n \colon \Gamma \to \mathrm{SL}_2(\ZZ/n\ZZ), \,\gamma\mapsto \gamma \Mod n.
$$

\begin{lemma}[Trace formula]\label{lem:inducedChar} Let $\Gamma$ be a subgroup of $ \mathrm{SL}_2(\mathbb{Z})$ and let $p$ be a prime such that $\pi_p \colon \Gamma \to \mathrm{SL}_2(\ZZ/p\ZZ)$ is onto. Then, writing $d(\gamma) = \mathrm{tr}(\gamma)^2-4$, we have
\begin{equation}
\mathrm{tr}(\lambda_p^0(\gamma)) = \begin{cases}
p &\text{ if $\gamma=\pm I \Mod p $ }\\
\left( \frac{d(\gamma)}{p} \right) &\text{ else}.
\end{cases}
\end{equation}
In particular, $\dim(\lambda_p^0 ) = p$.
\end{lemma}

\begin{proof}
For the rest of the proof we write $G_p = \mathrm{SL}_2(\ZZ/p\ZZ)$. Observe that $\Gamma_0(p)$ is equal to the pre-image $\pi_p^{-1}(B_p)$ of the subgroup of upper triangular matrices
$$
B_p \defeq \left\{ \pmat{\ast}{\ast}{0}{\ast} \right\} \leqslant G_p. 
$$
Since $\pi_p \colon\Gamma \to G_p$ is surjective we have $X_p \defeq B_p \backslash G_p \cong \Gamma_0(p)\backslash \Gamma$. Therefore, $\lambda_p = \mathrm{ind}_{\Gamma_0(p)}^{\Gamma}(\textbf{1}_{\Gamma_0(p)})$ is equivalent to the regular representation on the space $\ell^2(X_p)$ of functions $f\colon X_p \to \mathbb{C}$ endowed with the standard inner product. This representation, which by abuse of notation we also denote by $\lambda_p$, is defined for all $f\in \ell^2(X_p)$, $ \gamma\in G_p$ and $x\in X_p$ by
$$
\lambda_p(\gamma)f(x) = f(x\gamma).
$$
In this view, $\lambda_p^0$ is the restriction of $\lambda_p$ to the subspace $\ell^2_0(X_p)$ consisting of functions $f\in \ell^2(X_p)$ with $\sum_{x\in X_p} f(x)=0.$ In particular, since $\lambda_p = \lambda_p^0 \oplus \textbf{1}_\Gamma$, the trace satisfies
$$
\mathrm{tr}( \lambda_p^0(\gamma) ) = \mathrm{tr}( \lambda_p(\gamma) ) -1. 
$$
The proof rests on the following two claims:
\begin{enumerate}[(i)]
\item \label{claim1} The map sending a line $\ell$ in $\mathbb{F}_p^2$ to $\mathrm{Stab}_{G_p}(\ell)$ is a bijection from $\mathbb{P}(\mathbb{F}_p^2)$ to the set of conjugates of $B_p$ in $G_p$;
\item \label{claim2} The normalizer of $B_p$ in $G_p$ equals $B_p$, so we can identify to the set of conjugates of $B_p$ in $G_p$ with $X_p=B_p \backslash G_p$. 
\end{enumerate}
We proceed in two steps: first we show that for all $\gamma\in G_p$ the trace $\mathrm{tr}( \lambda_p(\gamma) )$ is equal to the number $L(\gamma)$ of lines $\ell \subseteq \mathbb{F}_p^2$ such that $\gamma\ell = \ell$, and then we compute $L(\gamma)$.\\

\noindent
\textbf{Step 1:} Let $\delta_x$ be the function on $X_p$ taking $1$ at $x$ and $0$ elsewhere. Note that $\{ \delta_x : x\in X_p\}$ is an orthonormal basis for $\ell^2(X_p)$. Thus, $\mathrm{tr}( \lambda_p(\gamma) )$ is equal to the number of fixed points of the action of $\gamma$ on $X_p$. Note that $B_p g \gamma = B_p g$ if and only if $\gamma$ belongs to $g^{-1}B_p g$. By Claim \eqref{claim1}, $X_p$ corresponds to the set of conjugates of $B_p$ in $G_p$, so $\mathrm{tr}( \lambda_p(\gamma) )$ is equal to the number of conjugates of $B_p$ containing $\gamma$. By Claim \eqref{claim1}, this is equal to $L(\gamma)$.\\

\noindent
\textbf{Step 2:} Note that the discriminant of the characteristic polynomial of $\gamma$ is equal to $d(\gamma)=\mathrm{tr}(\gamma)^2-4.$ If $d(\gamma)=0$, then $\gamma$ has a double eigenvalue $\lambda$, namely $\lambda=\pm 1.$ The eigenspace $E_\lambda$ has dimension either $2$ or $1$. If $\dim(E_\lambda)=2$, then $\gamma = \pm I$, so $\gamma$ fixes every line in $\mathbb{P}(\mathbb{F}_p^2)$. Thus, in this case we have $L(\gamma) = \# \mathbb{P}(\mathbb{F}_p^2) = \frac{p^2-1}{p-1}=p+1$, so $\mathrm{tr}( \lambda_p^0(\gamma) )=L(\gamma)-1 = p.$ If $\dim(E_\lambda)=1$, then $\gamma$ only fixes the line $E_\lambda$, whence $\mathrm{tr}( \lambda_p^0(\gamma) )=L(\gamma)-1 = 0 = \left( \frac{d(\gamma)}{p} \right).$

Now suppose alternatively that $d(\gamma)\neq 0$. If $d(\gamma)$ is a square in $\mathbb{F}_p^\times$, then it has two distinct eigenvalues. In this case $L(\gamma)=2$, since $\gamma$ fixes exactly its two eigenspaces (which are different lines). If $d(\gamma)$ is not a square in $\mathbb{F}_p^\times$, then $\gamma$ has no eigenvalue in $\mathbb{F}_p$, in which case $L(\gamma)=0$. In both cases, we have 
$$
L(\gamma)-1 = \left( \frac{d(\gamma)}{p} \right).
$$
This completes the proof.
\end{proof}

\begin{remark}
There is an alternative proof of Lemma \ref{lem:inducedChar} based on the Frobenius induction formula (also known as Mackey formula). Letting $s,t\in G_p $ denote the elements
$$
s = \pmat{0}{-1}{1}{0}, \quad t = \pmat{1}{1}{0}{1},
$$
one can verify by direct computation that the $p+1$ elements
\begin{equation}\label{explicit_repre}
I = t^0,\, t,\, t^2,\, \dots,\, t^{p-1},\,s
\end{equation}
provide an explicit set of representatives for the (left or right) cosets of $B$ in $ G_p $. If $\pi_p \colon \Gamma \to G_p$ is onto, then the representation $\lambda_p$ is equivalent to $\mathrm{Ind}_{B_p}^{G_p}(\textbf{1}_{B_p})$, so we may express its trace in terms of these representatives as follows:
\begin{equation}\label{frobmackey}
\mathrm{tr}(\lambda_p(\gamma)) =\textbf{1}_{B_p}(s^{-1}\gamma s)  + \sum_{j=0}^{p-1} \textbf{1}_{B_p}(t^{-j} \gamma t^j),
\end{equation}
where $\textbf{1}_{B_p}$ is indicator function of $B_p$. A somewhat tedious calculation then leads to the formula claimed in Lemma \ref{lem:inducedChar}.
\end{remark}

Also crucial is the following result:

\begin{lemma}\label{lem:qnorm}
Let $q \geqslant 2$ be an integer and let $\gamma\in \mathrm{SL}_2(\ZZ)$ be a hyperbolic element such that $\gamma\equiv \pm I \mod q$. Then we have
$$
\Vert \gamma\Vert >  \frac{q^2}{3}.
$$
\end{lemma}

\begin{proof}
Write
$$
\gamma = \pmat{a}{b}{c}{d} \in \mathrm{SL}_2(\mathbb{Z}).
$$
We use the following observation due to Sarnak--Xue \cite{SarnakXue}: if $\gamma\equiv \pm I \mod q$, then the trace $\mathrm{tr}(\gamma) = a+d $ satisfies the congruence
\begin{equation}\label{non_trivial_congreunce}
\tr(\gamma) \equiv \pm 2 \mod q^{2}.
\end{equation}
To see this, note that $\gamma\equiv \pm I \mod q$ implies that there are integers $ a', b', c', d'\in \ZZ $ such that
$$
\text{$ a = a'q \pm 1 $, $ b = b'q $, $ c = c'q $, and $ d = d'q \pm 1 $.}
$$
Furthermore, the relation $ad-bc = 1$ gives
\begin{equation}
1 = (a'd' -b'c') q^{2} \pm (a'+d')q + 1,
\end{equation}
which forces
$$
a'+d' = 0 \mod q.
$$
But this implies that
\begin{equation}\label{congrxuesarnak}
\mathrm{tr}(\gamma) = a+d = (a'+d' )q \pm 2 = \pm 2 \mod q^2
\end{equation}
as claimed. Now since $\gamma$ is hyperbolic by assumption we have $\vert \mathrm{tr}(\gamma)\vert > 2$, which combined with the congruence in \eqref{congrxuesarnak} implies that for all $q\geqslant 2$
\begin{equation}\label{lowerBoundForTrace}
\vert a + d\vert = \vert \mathrm{tr}(\gamma)\vert \geqslant  q^{2}-2 \geqslant  q^2/2.
\end{equation}
We deduce that for all $q\geqslant 2$
\begin{align*}
\Vert \gamma\Vert^2 &= a^2 + b^2 + c^2 + d^2\\
&= (a+d)^2 + (b-c)^2 - 2 \\
&\geqslant  (a+d)^2 - 2\\
&\geqslant  q^{4}/4 - 2\\
&\geqslant  q^4/8.
\end{align*}
Thus, $\Vert \gamma\Vert \geqslant  \sqrt{q^4/8} > q^2 /3 $, as claimed.
\end{proof}

We are now ready to finish the proof of Proposition \ref{prop:inducedSum}:

\begin{proof}[Proof of Proposition \ref{prop:inducedSum}]
Fix a finitely-generated subgroup $\Gamma\subset \mathrm{SL}_2(\ZZ)$ and a hyperbolic element 
$$
\gamma =\pmat{a}{b}{c}{d} \in \Gamma\smallsetminus \{I\}, \quad a,b,c,d\in \mathbb{Z},\quad ad-bc =1
$$
such that 
\begin{equation}\label{upperBoundNorm}
\Vert \gamma\Vert < \frac{x^2}{20}.
\end{equation}
Recall that $p\sim x$ means $p\in (\frac{x}{2}, x]$. Note that \eqref{upperBoundNorm} implies that $\gamma\neq \pm I \mod p$ for all $p\sim x$. If not, then Lemma \ref{lem:qnorm} would imply that
$$
\Vert \gamma\Vert \geqslant  \frac{p^2}{3} \geqslant  \frac{x^2}{12},
$$
contradicting \eqref{upperBoundNorm}. Furthermore, we know from Gamburd \cite{Gamburd1} that for primes $p$ large enough the reduction map $\pi_p \colon \Gamma \to \mathrm{SL}_2(\ZZ/p\ZZ)$ is onto. Hence, once $x$ is sufficiently large, Lemma \ref{lem:inducedChar} implies that for all primes $p\sim x$ we have
$$
\mathrm{tr}(\lambda_p^0(\gamma)) = \left( \frac{d(\gamma)}{p} \right).
$$
Thus,
$$
\sum_{\substack{ p\sim x \\ \text{$p$ prime} }} \log(p) \mathrm{tr}(\lambda_p^0(\gamma)) = \sum_{\substack{ p\sim x \\ \text{$p$ prime} }} \log(p)  \left( \frac{d(\gamma)}{p} \right). 
$$
Now we invoke Theorem \ref{thm:iwanieckowalski}. Since $\gamma\in \Gamma$ is hyperbolic we have $\mathrm{tr}(\gamma)\neq \pm 2$, so $d(\gamma) = \mathrm{tr}(\gamma)^2-4 \geqslant 1.$ Moreover, it is easy to verify that $d(\gamma)$ is either $0$ or $1$ modulo $4$ and that $d(\gamma) \ll\Vert \gamma\Vert^2 \ll x^4$, so we obtain
$$
\sum_{\substack{ p\sim x \\ \text{$p$ prime} }} \log(p) \left( \frac{d(\gamma)}{p} \right)  = O(x^{\frac{1}{2}} \log (x)^2),
$$
completing the proof.
\end{proof}

\subsection{Finishing the proof of Theorem \ref{Thm:AverageTheorem}}
Let us now complete the proof our main result. The main technical estimate in the proof is the following:

\begin{prop}[Sum of Hilbert-Schmidt norms]\label{prop:sum_of_HS_norm} 
Assume GRH for quadratic $L$-functions. Write $s = \sigma +it$. Then there are positive constants $x_0,c,\tau_0,C$, depending only on $\Gamma$, such that for all $x > x_0$ and for all $c x^{-2} < \tau < \tau_0$ we have
\begin{equation}
\sum_{\substack{ p\sim x \\ \text{$p$ prime}}} \Vert \mathcal{L}_{\tau, s, \lambda_p^0} \Vert_{\mathrm{HS}}^2  \ll e^{C \vert t\vert}(C\tau)^{2\sigma} \left( \tau^{-\delta} x^2 + x^{\frac{1}{2}} \log(x)^2 \tau^{-2\delta}\right).
\end{equation}
The implied constant depends solely on $\Gamma.$
\end{prop}

Let us show how we can use this proposition to deduce Theorem \ref{Thm:AverageTheorem_recalled}. (Recall from the discussion in §\ref{sec:setup} that our main theorem follows from Theorem \ref{Thm:AverageTheorem_recalled}.) Recall that it suffices to bound the sum
\begin{equation}\label{desired_bound}
S(x,\sigma) \defeq \sum_{\substack{ p\sim x \\ \text{$p$ prime}}} N_p(\sigma).
\end{equation}
By combining Propositions \ref{prop:SumOfHilbertSchmidtNorms} and \ref{prop:sum_of_HS_norm} we obtain that there are constants $c>0$ and $C>0$ such that once $x$ and $K$ are large enough and $\tau >0$ is small enough, we have
\begin{align}
S(x,\sigma) &\ll K \max_{\substack{ \mathrm{Re}(s)\geqslant  \sigma-O(\frac{1}{K}) \\ \vert \mathrm{Im}(s)\vert \leqslant  O(K) }} \left( \sum_{\substack{ p\sim x \\ \text{$p$ prime}}} \Vert \mathcal{L}_{\tau, s, \lambda_p^0} \Vert_{\mathrm{HS}}^2 \right) +  x^2 \tau^{K} \nonumber\\
&\ll e^{O(K)} (C\tau)^{2\sigma - O(\frac{1}{K})} \left( \tau^{-\delta} x^2 + x^{\frac{1}{2}} \log(x)^2 \tau^{-2\delta}\right) + x^2 \tau^{K}, \label{lastterm}
\end{align}
provided $\tau > c x^{-2}$. It remains to choose $K$ and $\tau$ optimally. This may be done by taking $K = (\log x)^{1/2}$, say, and 
$$
\tau = C^{-1} x^{-\frac{3}{2\delta}}.
$$
Observe that the required condition $\tau > c x^{-2}$ is satisfied as long as $\delta > \frac{3}{4}$ and $x$ is sufficiently large. Inserting these choices into \eqref{lastterm} we obtain that for any $x$ large enough and for any $\epsilon > 0$
$$
S(x,\sigma) \leqslant C(\epsilon,\Gamma) x^{1 - \frac{3}{\delta}(\sigma - \frac{5}{6}\delta) + \epsilon}.
$$
Note that we have used the trivial bounds $x^2 \tau^{K} \ll_\epsilon 1$, $\log(x)^2 \ll_\epsilon x^\epsilon$, $ e^{O(K)} \ll_\epsilon x^{\epsilon}$, and $\tau^{-O(\frac{1}{K})} \ll_\epsilon x^\epsilon$. This establishes Theorem \ref{Thm:AverageTheorem_recalled}. Now we give the proof of Proposition \ref{prop:sum_of_HS_norm}.

\begin{proof}[Proof of Proposition \ref{prop:sum_of_HS_norm}]
By Lemma \ref{lem:HSnorm} we can write down the following formula for the Hilbert--Schmidt norm of the operator $\mathcal{L}_{\tau,s,\lambda_p^0}$:
$$
\Vert \mathcal{L}_{\tau,s,\lambda_p^0} \Vert_{\mathrm{HS}}^2 = \sum_{b\in \mathcal{A}}\sum_{\substack{ \textbf{a},\textbf{b}\in Y(\tau) \\ \textbf{a},\textbf{b}\to b }}   \mathrm{tr}\left(\lambda_p^0( \gamma_{\textbf{a}}^{-1} \gamma_{\textbf{b}} )\right) \mathcal{I}_{\textbf{a},\textbf{b}}^{(b)},
$$
where
$$
\mathcal{I}_{\textbf{a},\textbf{b}}^{(b)} = \int_{D_b} \gamma'_{\textbf{a}}(z)^s  \overline{\gamma'_{\textbf{b}}(z)^s } B_D(\gamma_{\textbf{a}}(z), \gamma_{\textbf{b}}(z)) \dvol(z).
$$
Multiplying this formula by $\log (p)$ and summing over all primes in $(\frac{x}{2},x]$ gives
\begin{align}
\sum_{\substack{ p\sim x \\ \text{$p$ prime}}} \Vert \mathcal{L}_{\tau,s,\lambda_p^0}\Vert_{\mathrm{HS}}^2 &\leqslant \sum_{\substack{ p\sim x \\ \text{$p$ prime}}} \log(p) \Vert \mathcal{L}_{\tau,s,\lambda_p^0}\Vert_{\mathrm{HS}}^2 \label{eq:sumOverP} \\ 
&= \sum_{b\in \mathcal{A}}\sum_{\substack{ \textbf{a},\textbf{b}\in Y(\tau) \\ \textbf{a},\textbf{b}\to b }} \left( \sum_{\substack{ p\sim x \\ \text{$p$ prime}}}  \log(p)\mathrm{tr}\left(\lambda_p^0( \gamma_{\textbf{a}}^{-1} \gamma_{\textbf{b}} )\right) \right) \mathcal{I}_{\textbf{a},\textbf{b}}^{(b)}. \nonumber
\end{align}

Clearly, for the diagonal terms $\textbf{a}=\textbf{b}$ we have $\mathrm{tr}\left(\lambda_p^0( \gamma_{\textbf{a}}^{-1} \gamma_{\textbf{b}} )\right) = \mathrm{tr}(\lambda_p^0(I)) = p$, so by the prime number theorem we obtain
\begin{equation}
\sum_{\substack{ p\sim x \\ \text{$p$ prime}}} \log(p) \mathrm{tr}\left(\lambda_p^0( \gamma_{\textbf{a}}^{-1} \gamma_{\textbf{b}} )\right) = \sum_{\substack{ p\sim x \\ \text{$p$ prime}}} \log(p) p = O(x^2).
\end{equation}

Now we focus on the non-diagonal terms $\textbf{a}\neq \textbf{b}$. Here we appeal to Proposition \ref{prop:sum_of_HS_norm}. Recall from Lemma \ref{lem:YZ} that for all $\textbf{a}\in Y(\tau)$
$$
\Vert \gamma_\textbf{a}\Vert \leqslant  C \tau^{-1/2}
$$
for some constant $C=C(\Gamma)>0$. Thus, using the fact that the Frobenius norm is sub-multiplicative, we obtain for all $\textbf{a},\textbf{b}\in Y(\tau)$ 
$$
\Vert \gamma_{\textbf{a}}^{-1} \gamma_{\textbf{b}} \Vert \leqslant  \Vert \gamma_{\textbf{a}} \Vert \Vert \gamma_{\textbf{b}} \Vert < C^2\tau^{-1}.
$$
Therefore, taking $\tau > 20 C^2 x^{-2}$ gives
\begin{equation}
\Vert \gamma_{\textbf{a}}^{-1} \gamma_{\textbf{b}} \Vert < \frac{1}{20}x^2.
\end{equation}
Note further that $\textbf{a}\neq \textbf{b}$ implies that the element $ \gamma_{\textbf{a}}^{-1} \gamma_{\textbf{b}}$ is not the identity. Hence $ \gamma_{\textbf{a}}^{-1} \gamma_{\textbf{b}}$ is hyperbolic since the only non-hyperbolic element in $\Gamma$ is the identity. Thus, conditional on GRH for quadratic $L$-functions, Proposition \ref{prop:inducedSum} gives
\begin{equation}
\left\vert \sum_{\substack{ p\sim x \\ \text{$p$ prime}}} \log(p) \mathrm{tr}\left(\lambda_p^0( \gamma_{\textbf{a}}^{-1} \gamma_{\textbf{b}} )\right) \right\vert \ll x^{\frac{1}{2}} \log(x)^2.
\end{equation}
Inserting this back into \eqref{eq:sumOverP} yields
\begin{align*}
\sum_{\substack{ p\sim x \\ \text{$p$ prime}}} \Vert \mathcal{L}_{\tau,s,\lambda_p^0}\Vert_{\mathrm{HS}}^2 &\ll x^2 \sum_{b\in \mathcal{A}} \sum_{\substack{\textbf{a}\in Y(\tau)\\ \textbf{a}\to b}}  \mathcal{I}_{\textbf{a},\textbf{a}}^{(b)} + 
\sum_{b\in \mathcal{A}} \sum_{\substack{\textbf{a},\textbf{b}\in Y(\tau)\\ \textbf{a},\textbf{b}\to b}} \left\vert \sum_{\substack{ p\sim x \\ \text{$p$ prime}}} \log(p) \mathrm{tr}\left(\lambda_p^0( \gamma_{\textbf{a}}^{-1} \gamma_{\textbf{b}} )\right) \right\vert \vert \mathcal{I}_{\textbf{a},\textbf{b}}^{(b)}\vert\\
&\ll x^2 \sum_{b\in \mathcal{A}} \sum_{\substack{\textbf{a}\in Y(\tau)\\ \textbf{a}\to b}}  \mathcal{I}_{\textbf{a},\textbf{a}}^{(b)} + x^{\frac{1}{2}} \log(x)^2 \sum_{b\in \mathcal{A}} \sum_{\substack{\textbf{a},\textbf{b}\in Y(\tau)\\ \textbf{a},\textbf{b}\to b}} \vert \mathcal{I}_{\textbf{a},\textbf{b}}^{(b)}\vert.
\end{align*}
To estimate the remaining terms we use the bound in Lemma \ref{lem:HSnorm}: for all $b\in \mathcal{A}$ and $\textbf{a},\textbf{b}\in Y(\tau)$ with $\textbf{a},\textbf{b}\to b$
$$ 
\vert \mathcal{I}_{\textbf{a},\textbf{b}}^{(b)}\vert \ll (C \tau)^{2\sigma} e^{C \vert t\vert}
$$
for some $C = C(\Gamma) > 0$. Furthermore, by Lemma \ref{lem:YZ} we have
$$
\vert Y(\tau)\vert\ll \tau^{-\delta}.
$$
Inserting these bounds above gives
\begin{align*}
\sum_{\substack{ p\sim x \\ \text{$p$ prime}}} \Vert \mathcal{L}_{\tau,s,\lambda_p^0}\Vert_{\mathrm{HS}}^2 &\ll x^2 \sum_{b\in \mathcal{A}} \sum_{\substack{\textbf{a}\in Y(\tau)\\ \textbf{a}\to b}} (C \tau)^{2\sigma} e^{C \vert t\vert} + x^{\frac{1}{2}} \log(x)^2 \sum_{b\in \mathcal{A}} \sum_{\substack{\textbf{a},\textbf{b}\in Y(\tau)\\ \textbf{a},\textbf{b}\to b}} (C \tau)^{2\sigma} e^{C \vert t\vert}\\
&\ll x^2  \vert Y(\tau)\vert  (C \tau)^{2\sigma} e^{C \vert t\vert} + x^{\frac{1}{2}} \log(x)^2 \vert Y(\tau)\vert^2 (C \tau)^{2\sigma} e^{C \vert t\vert}\\
&\ll (C \tau)^{2\sigma} e^{C \vert t\vert} \left(  x^2 \tau^{-\delta} + x^{\frac{1}{2}} \log(x)^2 \tau^{-2\delta} \right).
\end{align*}
The proof of Proposition \ref{prop:sum_of_HS_norm} is complete.
\end{proof}

\normalem
\bibliography{hecke_shottky} 
\bibliographystyle{amsplain}

\end{document}